\documentclass{aims}
\usepackage{amsmath}
  \usepackage{paralist}
  \usepackage{graphics} %% add this and next lines if pictures should be in esp format
  \usepackage{epsfig} %For pictures: screened artwork should be set up with an 85 or 100 line screen
\usepackage{graphicx}  \usepackage{epstopdf}%This is to transfer .eps figure to .pdf figure; please compile your paper using PDFLeTex or PDFTeXify.
 \usepackage[colorlinks=true]{hyperref}
\hypersetup{urlcolor=blue, citecolor=red}

\def\currentvolume{7}
 \def\currentissue{4}
  \def\currentyear{2014}
   \def\currentmonth{August}
    \def\ppages{X--XX}
    \def\DOI{10.3934/dcdss.2014.7.xx}

\newtheorem{theorem}{Theorem}[section]
\newtheorem{corollary}{Corollary}

\newtheorem{lemma}[theorem]{Lemma}
\newtheorem{proposition}{Proposition}

\theoremstyle{definition}

\newtheorem{remark}{Remark}

\usepackage{amsfonts,amsthm,amssymb,amsxtra}

\newcommand{\R}{{\mathbb R}}
\renewcommand{\S}{{\mathbb S}}
\newcommand{\N}{{\mathbb N}}
\newcommand{\be}[1]{\begin{equation}\label{#1}}
\newcommand{\ee}{\end{equation}}
\renewcommand{\(}{\left(}
\renewcommand{\)}{\right)}
\newcommand{\finprf}{\unskip\null\hfill$\;\square$\vskip 0.3cm}

\renewcommand{\S}{\mathbb{S}}
\newcommand{\iS}[1]{\int_{\S^d}{#1}\;d\mu}

\newcommand{\nrm}[2]{\|{#1}\|_{\mathrm L^{#2}(\S^d)}}

\renewcommand{\H}{\mathrm H}
\newcommand{\Lap}{\Delta_{\S^d}}

\newcommand{\ix}[1]{\int_{-1}^1{#1}\;d\nu_d}

\newcommand{\nrmx}[2]{\|#1\|_{#2}}
\newcommand{\scal}[2]{\left\langle{#1},{#2}\right\rangle}

\newcommand{\nrmom}[2]{\|{#1}\|_{\mathrm L^{#2}(\Omega,d\mu)}}
\newcommand{\iom}[1]{\int_{\Omega}{#1}\;d\mu}

\newcommand{\e}{\mathsf e}

\renewcommand{\i}{\mathsf i}
\newcommand{\qp}{p}

\definecolor{darkgreen}{rgb}{0.2,0.7,0.1}

\title[Improved interpolation inequalities on the sphere]{Improved interpolation inequalities\\ on the sphere}

\author[J.~Dolbeault, M. J.~Esteban, M.~Kowalczyk and M.~Loss]{}%{Jean Dolbeault, Maria J.~Esteban,\\ Michal Kowalczyk, and Michael Loss}



\email{dolbeaul@ceremade.dauphine.fr}
\email{esteban@ceremade.dauphine.fr}
\email{kowalczy@dim.uchile.cl}
\email{loss@math.gatech.edu}

\begin{document}
\maketitle
\thispagestyle{empty}

\centerline{\scshape Jean Dolbeault and Maria J.~Esteban}
\medskip{\footnotesize\centerline{Ceremade (UMR CNRS 7534, Universit\'e Paris-Dauphine}
\centerline{Place de Lattre de Tassigny, 75775 Paris C\'edex~16, France}
}

\medskip\centerline{\scshape Micha\l~Kowalczyk}
\medskip
{\footnotesize\centerline{Departamento de Ingenieria Matem\'atica}
 \centerline{and Centro de Modelamiento Matem\'atico (UMI CNRS 2807)}
\centerline{Universidad de Chile, Casilla 170 Correo 3, Santiago, Chile}}

\medskip\centerline{\scshape Michael Loss}
\medskip{\footnotesize\centerline{School of Mathematics, Skiles Building}
\centerline{Georgia Institute of Technology, Atlanta GA 30332-0160, USA}}
%\bigskip
%%\centerline{(Communicated by the associate editor name)}

\begin{abstract} This paper contains a review of available methods for establishing improved interpolation inequalities on the sphere for subcritical exponents. Pushing further these techniques we also establish some new results, clarify the range of applicability of the various existing methods and state several explicit estimates.\end{abstract}

\maketitle
%\thispagestyle{empty}
%%%%%%%%%%%%%%%%%%%%%%%%%%%%%%%%%%%%%%%%%%%%%%%%%%%%%%%%%%%%%%%%%%%%%%
%\tableofcontents
%%%%%%%%%%%%%%%%%%%%%%%%%%%%%%%%%%%%%%%%%%%%%%%%%%%%%%%%%%%%%%%%%%%%%%
%%%%%%%%%%%%%%%%%%%%%%%%%%%%%%%%%%%%%%%%%%%%%%%%%%%%%%%%%%%%%%%%%%%%%%

%%%%%%%%%%%%%%%%%%%%%%%%%%%%%%%%%%%%%%%%%%%%%%%%%%%%%%%%%%%%%%%%%%%%%%
%%%%%%%%%%%%%%%%%%%%%%%%%%%%%%%%%%%%%%%%%%%%%%%%%%%%%%%%%%%%%%%%%%%%%%
\section{Introduction}\label{Sec:Intro}

On the $d$-dimensional sphere, for any real valued function $u$ in $\H^1(\S^d)$, let us consider the inequality
\be{Ineq:Interpolation}
d\,\e\le\i\quad\mbox{where}\quad\e:=\frac{\nrm up^2-\nrm u2^2}{p-2}\quad\mbox{and}\quad\i:=\nrm{\nabla u}2^2\,,
\ee
where $\mu$ is the normalized measure on $\S^d$ induced by the Euclidean measure on $\R^{d+1}$ and $p\in[1,2)\cup(2,2^*)$ with $2^*=\infty$ if $d\le 2$ and $2^*=\frac{2\,d}{d-2}$ if $d\ge3$. The case $p=2^*$ is also covered if $d\ge3$ and corresponds to Sobolev's inequality. When $p=1$, the inequality is equivalent to the Poincar\'e inequality. By taking the limit as $p\to 2$, we recover the logarithmic Sobolev inequality
\be{Ineq:LogSobolev}
\iS{u^2\,\log\(\frac{u^2}{\nrm u2^2}\)}\le\frac2d\,\nrm{\nabla u}2^2\quad\forall\,u\in\H^1(\S^d)\,.
\ee
The constant $d$ in \eqref{Ineq:Interpolation} is optimal: see for instance \cite{1302}. When $p=2$, it is consistent to define $\e$ as the l.h.s.~in~\eqref{Ineq:LogSobolev}, that is, $\e:=\iS{u^2\,\log\big(u^2/\nrm u2^2\big)}$. Hence for $p\in(1,2^*)$, the functionals $\e$ and $\i$ are respectively the generalized \emph{entropy} and \emph{Fisher information} functionals.

In this paper we are interested in improvements of \eqref{Ineq:Interpolation} and \eqref{Ineq:LogSobolev} in the subcritical range, that is, for \hbox{$1<p<2^*$}. By \emph{improved} inequality we mean an inequality of the form
\be{Ineq:InterpolationImproved}
d\,\nrm u2^2\;\Phi\!\(\frac{\e}{\nrm u2^2}\)\le\i\quad\forall\,u\in\H^1(\S^d)\,,
\ee
for some monotone increasing function $\Phi$ such that $\Phi(0)=0$, $\Phi'(0)=1$ and $\Phi(s)>s$ for any $s$. As a straightforward consequence we get a stability result. Indeed, let us set
\be{defpsi}
\Psi(s):=s-\Phi^{-1}(s)\,.
\ee
Hence we get
\[
\i-d\,\e=d\,\left[\nrm u2^2\;\Phi^{-1}\!\(\frac{\i}{d\,\nrm u2^2}\)-\e\right]+\,d\,\nrm u2^2\;\Psi\!\(\frac{\i}{d\,\nrm u2^2}\)\,,
\]
and thus,
\[
\i-d\,\e\ge\,d\,\nrm u2^2\;\Psi\!\(\frac{\i}{d\,\nrm u2^2}\)\,,
\]
where the inequality is a simple consequence of \eqref{Ineq:InterpolationImproved}. If, additionally, $\Psi$ is nondecreasing, then by reapplying \eqref{Ineq:InterpolationImproved} we find that
\be{Ineq:Stability}
\i-d\,\e\ge d\,\nrm u2^2\;(\Psi\circ\Phi)\!\(\frac{\e}{\nrm u2^2}\)\quad\forall\,u\in\H^1(\S^d)\,.
\ee
The function $\Psi\circ\Phi$ is nondecreasing, positive on $(0,\infty)$ and such that $(\Psi\circ\Phi)(0)=(\Psi\circ\Phi)'(0)=0$. Inequality~\eqref{Ineq:Stability} is a stability result since $\e$ controls a distance to the optimal functions, which are the constant functions. Our goal is to find the best possible function $\Phi$.

As an important application of the improved version of the inequalities, we can point \emph{stability} issues. Some straightforward consequences are:
\begin{enumerate}
\item the uniqueness of optimal functions when they exists, and a better characterization of the equality cases in the inequalities,
\item stability results in spectral theory with application to problems arising from quantum mechanics, like stability of matter,
\item some additional estimates in variational methods (improved convergence of sequences) with applications for instance when one uses Lyapunov-Schmidt reduction methods,
\item improved convergence rates in evolution problems.
\end{enumerate}
The last point is probably the most important in view of applications in physics. In various cases of interest, it allows to prove that before entering an asymptotic regime of, for instance, exponential decay, a system may have an initial regime with an even faster convergence.

%\medskip
Let us briefly review the literature and give some indications of our motivation and main results. Inequality~\eqref{Ineq:Interpolation} has been established in \cite{MR1271314,BV-V,MR1230930}. The limit case $p=2$ was known earlier: see for instance \cite{MR578933} in the case of the circle, and \cite{DEKL2012} for a detailed list of references.

In the case of compact manifolds other than the sphere, the estimates obtained by M.-F.~Bidaut-V\'eron and L.~V\'eron in \cite{BV-V} (also see \cite{MR1338283,MR1631581,MR1412446,MR1435336,MR1651226}) and by J.~Demange in \cite{Demange-PhD,MR2381156} can be improved at the leading order term by considering non-local quantities like in~\cite{1302}.

In the case of the sphere, the leading order term is determined by the constant~$d$ in \eqref{Ineq:Interpolation}. Looking for improvements therefore makes sense. J.~Demange observed in~\cite{MR2381156} that Inequality~\eqref{Ineq:Interpolation} can be improved when $p\in(2,2^*)$. Moreover, in~\cite{Demange-PhD} he noticed the existence of a free parameter. The first purpose of this paper is to clarify the range of the free parameter in the method and optimize on it, in order to get the best possible improvement with respect to that parameter. Recent contributions on nonlinear flows and Lyapunov functionals, or entropies, that can be found in \cite{DEKL2012,1302,DoEsLaLo2013}, are at the core of our method.

Refined convex Sobolev inequalities have been established in \cite{MR2152502} when $p\in(1,2)$ in the setting of interpolation inequalities involving a probability measure. For a simpler formulation, see \cite{pre05312043}. Our second purpose is to adapt the method to interpolation inequalities on the sphere and hence also cover the range $p\in(1,2)$. It is based on refined estimates of entropy decay for a linear heat flow, in the spirit of the Bakry-Emery method.

Our last contribution is inspired by \cite{ABD}. Under additional orthogonality conditions, we show that other (and in some cases better) improved inequalities can be established in the range $p\in(1,2)$. The method is based on hypercontractivity estimates for a linear heat flow and a spectral decomposition. It raises an intriguing open question on the possibility of obtaining improvements under orthogonality conditions in the range $p>2$.

%\medskip
Now let us explain the strategy of the paper.

\smallskip\noindent 1) Standard symmetrization techniques allow to decrease the function $\i$ while preserving the $\mathrm L^2$ and $\mathrm L^p$ norms, and the functional $\e$ as well. Thus, there are optimal functions for inequalities~\eqref{Ineq:InterpolationImproved} and~\eqref{Ineq:Stability} which depend only on the azimuthal angle on the sphere and the interpolation inequalities are therefore equivalent to one-dimensional inequalities for the $d$-ultraspherical operator, where $d$ can now be considered as a real parameter. Details are given in Section~\ref{Sec:Preliminaries}. We should point out that this first step simplifies the calculations needed to get an improved inequality, but it is not fundamental for our method. After a change of variables we are led to the following expressions
\[
\i=\ix{|f'|^2\;\nu}\quad\mbox{and}\quad\e=\frac{\(\ix{|f|^\qp}\)^\frac 2\qp- \ix{|f|^2}}{p-2}\,,
\]
where $\nu_d$ is a probability measure, and $\nu\geq 0$ is a smooth function (these functions are explicit but their exact form is irrelevant for the moment).
Then we define the self-adjoint \emph{ultraspherical} operator through the identity
\[
\ix{f_1'\,f_2'\;\nu}=-\ix{\mathcal L f_1\,f_2\;\nu}\,.
\]
The natural function space for our inequalities is the form domain of $\mathcal L$, that is
\[
\mathcal H:=\left\{f\in\mathrm L^2\big((-1,1),d\nu_d\big)\,:\,\ix{|f'|^2\;\nu}<\infty\right\}\,,
\]
and we shall denote by $\nrmx fq$ the $\mathrm L^q\big((-1,1),d\nu_d\big)$ norm of $f$.

\smallskip\noindent 2) The key to our approach is to combine the ideas of D.~Bakry and M.~Emery, \emph{i.e.}, take the derivative of $\,\i-\,d\,\e$ along some flow, with ideas that go back to B.~Gidas and J.~Spruck in \cite{MR615628} and that were later exploited by M.-F.~Bidaut-V\'eron and L.~V\'eron for getting rigidity results in nonlinear elliptic equations. An unessential but useful trick amounts to write the flow for $w$ with $f=w^\beta$ for some $\beta\in\R$, in the expressions for $\i$ and $\e$, as we shall see below.

Let us start with the case $\beta=1$. We consider the manifold
\[
\mathcal M_p=\left\{w\in\mathcal H\,:\,\nrmx wp=1\right\}\,.
\]
The main issue is to choose the right flow for our setting. We observe that if $\kappa=p-1$, then $\mathcal M_p$ is invariant under the action of the flow
\be{Eqn:Linear}
w_t=\L w+\kappa\,\frac{|w'|^2}w\,.
\ee
Notice that $g=w^p$ evolves according to the  equation $g_t=\L g$ and we shall therefore refer to the case associated with this equation as the \emph{linear case}, or the \emph{$1$-homogeneous case}. At the level of $w$, the equation is indeed $1$-homogeneous since $\lambda\,w$ is a solution to \eqref{Eqn:Linear} for any $\lambda>0$ if $w$ is a solution to \eqref{Eqn:Linear}.

Let
\[
2^\sharp:=\frac{2\,d^2+1}{(d-1)^2}
\]
and
\be{Eqn:gamma1}
\gamma_1:=\(\frac{d-1}{d+2}\)^2\,(p-1)\,(2^\#-p)\quad\mbox{if}\quad d>1\,,\quad\gamma_1:=\frac{p-1}3\quad\mbox{if}\quad d=1\,.
\ee
If $p\in[1,2)\cup(2,2^\sharp]$ and $w$ is a solution to \eqref{Eqn:Linear}, then
\[
\frac d{dt}\(\i-\,d\,\e\)\le-\,\gamma_1\ix{\frac{|w'|^4}{w^2}}\le -\,\gamma_1\,\frac{|\e'|^2}{1-\,(p-2)\,\e}\,.
\]
Recalling that $\e'=-\,\i$, we get a differential inequality
\[
\e''+\,d\,\e'\ge\gamma_1\,\frac{|\e'|^2}{1-\,(p-2)\,\e}\;,
\]
which after integration implies an inequality of the form
\[
d\,\Phi(\e(0))\le\i(0)\,.
\]
Details will be given in Sections~\ref{Sec:beta=1} and~\ref{Sec:Improved2}.

\smallskip\noindent 3) Now let us consider the case with a general $\beta$. The range of $p$'s for which an improved inequality is valid can be extended to any $p\in(1,2^*)$ by considering the nonlinear flow
\be{Eqn:Nonlinear}
w_t=w^{2-2\beta}\(\L w+\kappa\,\frac{|w'|^2}{w}\)
\ee
for some $\beta\ge1$, with $\kappa=\beta\,(p-2)+1$. Then $\mathcal M_{\beta p}$ is invariant under the flow and our improved functional inequality follows from the computation of $\frac d{dt}\(\i-\,d\,\e\)$ written for $f=w^\beta$, with the additional difficulty that $\e'$ now differs from $-\,\i$. Details on \eqref{Eqn:Nonlinear} will be given in Section~\ref{Sec:flow} and improved inequalities will be established in Section~\ref{Sec:Improved1}. The change of function $f=w^\beta$, for some parameter $\beta\in\R$, is convenient to compute $\frac d{dt}\(\i-\,d\,\e\)$ but also sheds light on the strategy used for proving rigidity according to the method of \cite{BV-V}. It moreover shows that the computations are equivalent and explains why a local bifurcation result from constant functions can be extended to a global uniqueness property. This is because the flow relates any initial datum to the constants through the monotonically decreasing quantity $\i-\,d\,\e$.

In his thesis, J.~Demange \cite{Demange-PhD} made a computation which is similar to ours. Compared to his approach, we work in a setting in which the change of function $f=w^\beta$ clarifies the relation of flow methods with rigidity results in nonlinear elliptic PDEs. We give an explicit range for the parameter $\beta$ and numerically observe that there is no optimal choice of $\beta$ valid for an arbitrary value of the entropy $\e$ when $p>2$. We also give a new result when $p\in(1,2)$ and it turns out that only $\beta=1$ can be used in that range.

\smallskip\noindent 4) Our last improvement is of a different nature. If $p\in(1,2)$ and assuming that the function $u$ is in the orthogonal of the eigenspaces associated with the lowest positive eigenvalues of the Laplace-Beltrami operator $\Lap$ on the sphere, using Nelson's hypercontractivity result, it is possible to obtain another improved inequality. Here we again use the linear flow~\eqref{Eqn:Linear} corresponding to $\beta=1$. Although it is a rather straightforward adaptation of~\cite{ABD}, such an approach raises an interesting open question. See Section~\ref{Sec:Spectral}.

\smallskip The nonlinear flow defined by~\eqref{Eqn:Nonlinear} and by~\eqref{Eqn:Linear} when $\beta=1$ is the main conceptual tool for our analysis. It is introduced in Section~\ref{Sec:flow}. Explicit stability results rely on generalized Csisz\'ar-Kullback-Pinsker type inequalities, which are detailed in Section~\ref{Sec:CKP}. Most of our statements, when written for the ultraspherical operator, are valid for $d$ taking real values and our proofs can be adapted without changes. However, for simplicity, we shall assume that $d$ is an integer throughout this paper, unless explicitly specified.

It is very likely that the improved inequalities presented in this paper are not optimal. What could be an \emph{optimal} $\Phi$ (or even if such a question really makes sense) is an open question, but at least we can construct a whole collection of such functions (depending on $p$ and $d$), which in most cases improve on previously known results. The strategy that we use to prove the improved inequalities builds up on some previous works, but the way we combine these ideas is new and suggests several directions for future investigations.

%\medskip
After all these preliminaries, we are now in position to state our results. Let us start by giving an expression of the function $\Phi$. For each $p\in (1,2)\cup (2, 2^\sharp)$ and $\gamma_1$ defined by \eqref{Eqn:gamma1}, let
\be{Eqn:varphi1}
\varphi_1(s):=\begin{cases}\displaystyle
\tfrac 2{\gamma_1+2\,(p-2)}\left[\big(1-\,(p-2)\,s\big)^{-\frac{\gamma_1}{2\,(p-2)}}-1+\,(p-2)\,s\right]\quad\mbox{if}\quad p\neq2-\frac{\gamma_1}{2}\,,\\[6pt]
\displaystyle\tfrac1{2-p}\,\big(1-\,(p-2)\,s\big)\;\log \big(1-\,(p-2)\,s\big)\quad\mbox{if}\quad p=2-\frac{\gamma_1}{2}\,,
\end{cases}
\ee
where we assume that $s$ is \emph{admissible}, that is, $s>0$ if $p\in(1,2)$ and $s\in(0,\frac1{p-2})$ if $p>2$. Next, we define
\be{Eqn:gamma}
\gamma(\beta):=-\(\frac{d-1}{d+2}\,(\kappa+\beta-1)\)^2+\,\kappa\,(\beta-1)+\,\frac d{d+2}\,(\kappa+\beta-1)\,.
\ee
We observe that $\gamma(1)=\gamma_1$, and that there exists some $\beta\in\R$ such that $\gamma(\beta)>0$ only if $p\ge1$ if $d=1$ or if $p\in[2,2^*]$ for any $d>1$. Then, we define $\mathfrak B(p,d)$ by
\be{Eqn:admbeta}\begin{array}{l}
\mathfrak B(p,d)=\left\{\beta\in\R\,:\,\gamma(\beta)>0\,,\;\beta\ge1\,,\;\mbox{and}\;\beta\le\tfrac2{4-p}\;\mbox{if}\;p<4\right\}\quad\mbox{if}\quad p>2\,,\\[6pt]
\mathfrak B(p,d)=\{1\}\quad\mbox{if}\quad1\le p\le2\,.
\end{array}
\ee
A more explicit description of the region $\mathfrak B(p,d)$ will be given in the Appendix. For each $\beta\in\mathfrak B(p,d)$ such that $\beta>1$, let
\be{Eqn:varphibeta}
\varphi_\beta(s):=\int_0^{s}\exp\left[\tfrac{\gamma(\beta)}{\beta\,(\beta-1)\,p}\(\(1\,-\,(p-2)\,z\)^{1-\delta(\beta)}-\(1\,-\,(p-2)\,s\)^{1-\delta(\beta)}\)\right]\,dz
\ee
with
\be{Eqn:delta}
\delta(\beta):=\frac{p-\,(4-p)\,\beta}{2\,\beta\,(p-2)}\,.
\ee
The reader is invited to check that $\lim_{\beta\to1_+}\varphi_\beta(s)=\varphi_1(s)$ for any admissible $s$. Finally we define
\be{Def:Phi}
\Phi(s):=\big(1+(p-2)\,s\big)\,\varphi\!\(\frac s{1+(p-2)\,s}\)\quad\mbox{where}\quad\varphi(s):=\sup_{\beta\in\mathfrak B(p,d)}\varphi_\beta(s)\,.
\ee
%-------------------------------------------------------------------------------------
\begin{theorem}\label{Thm:Main} Assume that one of the following conditions is satisfied:
\begin{enumerate}
\item[(i)] $d=1$ and $p\in(1,2)\cup(2,\infty)$,
\item[(ii)] $d=2$ and $p\in(1,2)\cup(2,9+4\sqrt3)$,
\item[(iii)] $d\ge3$ and $p\in(1,2)\cup(2,2^*)$.
\end{enumerate}
For any $u\in\H^1(\S^d)$ be such that $\nrm u2=1$, we have the inequality
\[
d\;\Phi\!\(\frac{\nrm up^2-1}{p-2}\)\le\nrm{\nabla u}2^2\,,
\]
where $\Phi$ is defined by \eqref{Def:Phi}. Moreover $\Phi(0)=0$, $\Phi'(0)=1$ and $\Phi''(s)>0$, for a.e.~$s>0$ if $p>2$ or for a.e.~$s\in(0,1/(2-p))$ if $p\in(1,2)$. If $p=2$, then for any $u\in\H^1(\S^d)$ the following \emph{improved logarithmic Sobolev} inequality
\[
\iS{u^2\,\log\(\frac{|u|^2}{\nrm u2^2}\)}\le\frac4{\gamma_1^*}\,\nrm u2^2\,\log\(1+\frac{\gamma_1^*}{2d}\,\frac{\nrm{\nabla u}2^2}{\nrm u2^2}\)
\]
holds with $\gamma_1^*=\frac{4\,d-1}{ (d+2)^2}$.\end{theorem}
%-------------------------------------------------------------------------------------

The reader interested in best constants in logarithmic Sobolev inequalities as in Theorem~\ref{Thm:Main} is invited to refer to \cite{MR1961176,MR1988630}.
In the case $p\neq 2$, a first consequence of the above improved inequality, compared to the standard inequality $\i\ge d\,\e$, is that $\i-\,d\,\e$ is not only nonnegative, but that it actually measures a distance to the constants in the homogeneous Sobolev norm. With previous notations, one can indeed state that
\[
\nrm{\nabla u}2^2-\,\frac d{p-2}\,\Big[\nrm up^2-1\Big]\ge d\,\Psi\(\tfrac 1d\,\nrm{\nabla u}2^2\)
\]
for any $u\in\H^1(\S^d)$ such that $\nrm u2=1$, where $\Psi$ is defined by \eqref{defpsi}. We can rephrase this result without normalization as the following corollary.
%---------------------------------------------------------------------
\begin{corollary}
Assume that $p\in(1,2)\cup(2,2^*)$. With the notations of Theorem~\ref{Thm:Main}, we have
\[
\nrm{\nabla u}2^2-\,\frac d{p-2}\,\Big[\nrm up^2-\nrm u2^2\Big]\ge d\,\nrm u2^2\,\Psi\!\(\frac{\nrm{\nabla u}2^2}{d\,\nrm u2^2}\)\,,
\]
where $\Psi$ is defined by \eqref{defpsi} in terms of $\Phi$, and $\Phi$ is given by \eqref{Def:Phi}.
\end{corollary}
%---------------------------------------------------------------------

Now we can use a generalized Csisz\'ar-Kullback-Pinsker inequality to get an estimate of the distance to the constants using a standard Lebesgue norm. We shall distinguish three cases: (i) $p\in(1,2)$, (ii) $p\in(2,4)$ , and (iii) $p\ge4$. Let us define $q(p):=2/p$, $p/2$ and $p/(p-2)$ in cases (i), (ii) and (iii) respectively. Assume also that $r(p):=p$, $2$ and $p-2$ in cases (i), (ii) and (iii) respectively and let $s(p):=\max\{2,p\}$.
%---------------------------------------------------------------------
\begin{proposition}\label{prop:CK} Assume that $p\in(1,2)\cup(2,\infty)$. With the above notations and $(q,r,s)=(q(p),r(p),s(p))$, there exists a positive constant $\mathsf C$, depending only on $p$, such that for any $u\in\mathrm L^1\cap\mathrm L^p(\Omega,d\mu)$, we have
\[
\frac1{p-2}\,\Big[\nrm up^2-\nrm u2^2\Big]\ge\mathsf C\,\nrm us^{2\,(1-r)}\,\nrm{u^r-\bar u^r}q^2
\]
where $\bar u=\nrm ur$.\end{proposition}
%---------------------------------------------------------------------
The case $p=2$ appears as a limit case and corresponds to the standard Csisz\'ar-Kullback-Pinsker inequality. Details on this limit and an explicit estimate of $\mathsf C$ for all $p\ge1$ will be given in Section~\ref{Sec:CKP} (see Corollary~\ref{cor:CK}). Notice that $\mathsf C=1$ when $p=1$, and the inequality in Proposition~\ref{prop:CK} is thus equivalent to a Poincar\'e inequality. A direct consequence of Proposition~\ref{prop:CK} and Inequality~\eqref{Ineq:Stability} is the following \emph{stability} result.
%---------------------------------------------------------------------
\begin{corollary}\label{cor:stab} With the notations of Theorem~\ref{Thm:Main} and Proposition~\ref{prop:CK}, we have
\begin{multline*}
\nrm{\nabla u}2^2-\,\frac d{p-2}\,\Big[\nrm up^2-\nrm u2^2\Big]\\\ge d\,\nrm u2^2\,(\Psi\circ\Phi)\!\(\mathsf C\,\tfrac{\nrm us^{2\,(1-r)}}{\nrm u2^2}\,\,\nrm{u^r-\bar u^r}q^2\)\quad\forall\,u\in\H^1(\S^d)\,.
\end{multline*}\end{corollary}
%---------------------------------------------------------------------

%\medskip
 Now let us turn our attention to the statements corresponding to the last class of improvements studied in this paper. The range of $p$ is now restricted to $[1,2)$ and we consider the linear flow~\eqref{Eqn:Linear}.

In \cite{MR954373}, W.~Beckner gave a method to prove interpolation inequalities between logarithmic Sobolev and Poincar\'e inequalities in the case of a Gaussian measure. The method extends to the case of the sphere as was proved in \cite{DEKL2012}, in the range $\qp\in[1,2)$, with optimal constants. For further considerations on inequalities that interpolate between Poincar\'e and logarithmic Sobolev inequalities, we refer to \cite{MR2152502,ABD,Bakry-Emery85,MR772092,MR2766956,MR2609029,MR2081075,MR1796718} and references therein.

Our purpose is to obtain an improved estimate of the optimal constant ${\tt C}_p$ in
\be{Eqn:BecknerExtended}
\iS{|\nabla u|^2}\ge {\tt C}_p\left[\iS{|u|^2}-\(\iS{|u|^\qp}\)^{2/\qp}\right]\quad\forall\;u\in\H^1(\S^d,d\mu)\,.
\ee
Without any constraint on $u$, we have ${\tt C}_p=\frac d{2-\qp}$, and it is natural to expect that this constant will be improved when we impose additional constraints on the set of admissible $u$'s. It the present case we will assume that $u$ is in the orthogonal complement of the finite dimensional subspace spanned by the spherical harmonics corresponding to the lowest positive eigenvalues of the Laplace-Beltrami operator $\Lap$ on the sphere. Under this additional hypothesis we will obtain an improvement of the estimate \eqref{Eqn:BecknerExtended}, similar to what was done \cite{ABD}, in a different setting.

Let us introduce some notations and recall some known results. We consider $\Lap$ as an operator on $\mathrm L^2(\S^d,d\mu)$ with domain $\H^2(\S^d,d\mu)$, whose eigenvalues are
\[
\lambda_k=k\,(k+d-1)\quad\forall\,k\in\N\,,
\]
and denote by $E_k$ the corresponding eigenspaces. Recall that
\[
\mathrm{dim}(E_k)=\frac{(k+d-2)!}{k!\,(d-1)!}\,(2\,k+d-1)\quad\forall\,k\ge0\,,
\]
according to \cite[Corollary C-I-3, page 162]{MR0282313}. Notice that $E_0$ is generated by the function $1$. Finally, for $k=1, 2, \dots$ let us define constants
\[
\alpha_k:=\frac1d\,\lambda_{k+1}\,,
\]
and functions
\[
\Phi_k(s)=\max\left\{\frac{\alpha_k}{1-(p-1)^{\alpha_k}} (1-s), \frac{\log s}{\log (p-1)}\right\}\,.
\]
With these notations we have:
%---------------------------------------------------------------------
\begin{theorem}\label{Thm:Main2} Assume that $p\in(1,2)$. If $u\in\H^1(\S^d)$ is such that $\nrm u2=1$ and
\be{Eqn:Orthogonality}
\iS{u\,e}=0\quad\forall\,e\in E_j\,,\quad j=1\,,2\ldots k\,,
\ee
then we have
\[
\nrm{\nabla u}2^2\ge d\,\Phi_k\(\nrm up^2\)\,.
\]
\end{theorem}
%---------------------------------------------------------------------
We may notice that $\alpha\mapsto\frac\alpha{1-(p-1)^\alpha}$ is an increasing function of $\alpha>1$ and thus larger that $\frac{\alpha_0}{1-(p-1)^{\alpha_0}}=\frac1{2-p}$ since $\alpha_0=1$. Hence for any $k\ge1$, we have
\[
\frac{d\,\alpha_k}{1-(p-1)^{\alpha_k}}>\frac d{2-p}\,,
\]
thus showing that we have achieved a strict improvement of the constant compared to the one in \eqref{Eqn:BecknerExtended} (which is optimal without further assumption).

Moreover, $\lim_{s\to0_+}\Phi_k(s)=+\infty$, thus showing that when the orthogonality conditions are satisfied, the estimate of Theorem~\ref{Thm:Main2} is strictly better than the one of Theorem~\ref{Thm:Main}. See Fig.~\ref{Fig5-2} for an illustration.

%%%%%%%%%%%%%%%%%%%%%%%%%%%%%%%%%%%%%%%%%%%%%%%%%%%%%%%%%%%%%%%%%%%%%%
%%%%%%%%%%%%%%%%%%%%%%%%%%%%%%%%%%%%%%%%%%%%%%%%%%%%%%%%%%%%%%%%%%%%%%
\section{Symmetrisation and the ultraspherical operator}\label{Sec:Preliminaries}

As we have pointed out in the introduction, the improved inequality on the $d$-dimensional sphere can be reduced to the inequalities for functions depending only on the azimuthal angle. The interested reader can refer to \cite{DEKL2012}. Alternatively, let us give a sketch of a proof for completeness. More details on the stereographic projection can be found in \cite[Appendix B.3]{DoEsLa2012}

We denote by $\xi=(\xi_0,\,\xi_1,\ldots\xi_d)$ the coordinates of an arbitrary point in the unit sphere $\S^d\subset\R^{d+1}$. Consider the stereographic projection $\Sigma:\S^d\setminus\{\mathrm N\}\to\R^d$, where $\mathrm N$ denotes the \emph{North Pole}, that is the point in $\S^d$ corresponding to $\xi_d=1$. Here $\Sigma(\xi)=x\,$ means $\,x=(1-\xi_d)^{-1}\,(\xi_0,\,\xi_1,\ldots\xi_{d-1})$, and to any $u\in\H^1(\S^d)$ we associate a function $\sigma[u]:=v\in\H^1(\R^d)$ such that
\[
u(\xi)=(1-\xi_d)^{-\frac{d-2}2}\,v(x)\quad\forall\,\xi\in\S^d\quad\mbox{and}\quad x=\Sigma(\xi)\,.
\]
An elementary computation shows that $\nrm{\nabla u}2^2+\frac{d\,(d-2)}4\,\nrm u2^2=\|\nabla v\|_{\mathrm L^2(\R^d)}^2$, $\nrm u2^2=\|\,4\,(1+|x|^2)^{-2}\,v\,\|_{\mathrm L^2(\R^d)}^2$ and $\nrm up^2=\|(2/(1+|x|^2))^{1+\frac dp-\frac d2}v\,\|_{\mathrm L^p(\R^d)}^2$. To $v$ we may apply the standard Schwarz symmetrization and denote the symmetrized function by $v_*$. Let us define
\[
u^*:=\sigma^{-1}\big[(\sigma[u])_*\big]\,.
\]
%---------------------------------------------------------------------
\begin{lemma}\label{Lem:Sym} Assume that $u\in\H^1(\S^d)$. Then we have
\[
\i[u]=\iS{|\nabla u|^2}\ge\i[u^*]
\]
and for any $q\in[1,2^*]$, $\nrm uq=\nrm{u^*}q$, so that
\[
\e[u]=\e[u^*]\,.
\]\end{lemma}
%---------------------------------------------------------------------
This symmetry result is a kind of folklore in the literature and we can quote \cite{MR0402083,MR717827,MR1164616} for various related results. Details of the proof are left to the reader. As a straightforward consequence, $\inf\i[u]/\Phi(\e[u])$ is achieved by functions depending only on $\xi_d$, or on the azimuthal angle $\theta$.

Thus, to prove the inequality in Theorem \ref{Thm:Main} it suffices to prove it for
\[
\e= \frac1{p-2}\left[\left(\int_0^\pi |v(\theta)|^\qp\;d\sigma\)^\frac 2\qp-\left(\int_0^\pi|v(\theta)|^2\;d\sigma\right)\right], \quad \i =\int_0^\pi|v'(\theta)|^2\;d\sigma,
\]
and for any function $v\in\H^1([0,\pi],d\sigma)$, where
\[
d\sigma(\theta):=\frac{(\sin\theta)^{d-1}}{Z_d}\,d\theta\quad\mbox{with}\quad Z_d:=\sqrt\pi\,\frac{\Gamma(\tfrac d2)}{\Gamma(\tfrac{d+1}2)}\,.
\]
The change of variables $x=\cos\theta$, $v(\theta)=f(x)$ allows to rewrite the expressions for $\e$ and $\i$:
\[\label{Ultraspherical}
\e = \frac1{\qp-2}\left[ \(\ix{|f|^\qp}\)^\frac 2\qp-\ix{|f|^2}\ge\right], \quad \i =\int_{-1}^1|f'|^2\;\nu\;d\nu_d\,,
\]
where $d\nu_d$ is the probability measure defined by
\[
\nu_d(x)\,dx=d\nu_d(x):=Z_d^{-1}\,\nu^{\frac d2-1}\,dx\quad\mbox{with}\quad\nu(x):=1-x^2\;,\quad Z_d=\sqrt\pi\,\frac{\Gamma(\tfrac d2)}{\Gamma(\tfrac{d+1}2)}\,.
\]

We consider the space $\mathrm L^2((-1,1),d\nu_d)$ equipped with the scalar product
\[
\scal{f_1}{f_2}=\ix{f_1\,f_2}\,,
\]
and recall that the \emph{ultraspherical} operator is given by $\scal{f_1}{\L f_2}=-\ix{f_1'\,f_2'\;\nu}$. Explicitly we have:
\[
\L f:=(1-x^2)\,f''-d\,x\,f'=\nu\,f''+\frac d2\,\nu'\,f'.
\]
With these notations, for any positive smooth function $u$ on $(-1,1)$, Inequalities~\eqref{Ineq:Interpolation} and \eqref{Ineq:LogSobolev} can be rewritten as
\begin{eqnarray*}
&&-\scal f{\L f}=\ix{|f'|^2\;\nu}\ge d\,\frac{\nrmx f\qp^2-\nrmx f2^2}{\qp-2}\,,\\
&&-\scal f{\L f}\ge\frac d2\ix{|f|^2\,\log\(\frac{|f|^2}{\nrmx f2^2}\)}\,,
\end{eqnarray*}
if $p\in[1,2)\cup(2,2^*)$ and $p=2$, respectively. In the framework of the ultraspherical operator, the parameter $d$ can be considered as a positive real parameter, with critical exponent $2^*=\frac{2\,d}{d-2}$ if $d>2$, and $2^*=\infty$ if $d\le2$. We refer to \cite{MR674060,MR1260331,MR2641798,MR1231419,MR1971589,MR1435336,MR1651226,DEKL2012} for more references. The next lemma gives two elementary but very useful identities:
%---------------------------------------------------------------------
\begin{lemma}\label{lem:identies}
For any positive smooth function $u$ on $(-1,1)$, we have
\begin{eqnarray*}
&&\ix{(\L w)^2}=\ix{|w''|^2\;\nu^2}+d\ix{|w'|^2\;\nu}\,,\\
&&\scal{\frac{|w'|^2}w\;\nu}{\L w}=\frac d{d+2}\ix{\frac{|w'|^4}{w^2}\;\nu^2}-\,2\,\frac{d-1}{d+2}\ix{\frac{|w'|^2\,w''}w\;\nu^2}\,.
\end{eqnarray*}\end{lemma}
%---------------------------------------------------------------------

%%%%%%%%%%%%%%%%%%%%%%%%%%%%%%%%%%%%%%%%%%%%%%%%%%%%%%%%%%%%%%%%%%%%%%
%%%%%%%%%%%%%%%%%%%%%%%%%%%%%%%%%%%%%%%%%%%%%%%%%%%%%%%%%%%%%%%%%%%%%%
\section{Flows}\label{Sec:flow}

As we explained in the introduction, the main step in our methodology is to take the derivative of the expression $\i-\,d\,\e$ along the manifold $\mathcal M_p$ of functions whose $\mathrm L^p$ norm is equal to $1$, following the evolution given by well-chosen flows. In this section we will describe a special flow which leaves this manifold invariant and we will carry out the computation of the derivative.
%%%%%%%%%%%%%%%%%%%%%%%%%%%%%%%%%%%%%%%%%%%%%%%%%%%%%%%%%%%%%%%%%%%%%%
\subsection{The \texorpdfstring{$1$}1-homogeneous case (\texorpdfstring{$\beta=1$)}{beta=1}}\label{Sec:beta=1}
On $(-1,1)$, let us consider the flow defined by \eqref{Eqn:Linear}, that is, $w_t=\L w+\kappa\,\frac{|w'|^2}w$, and notice that
\[
\frac d{dt}\ix{w^p}=(\kappa-p+1)\ix{w^{p-2}\,|w'|^2\,\nu}\,,
\]
so that $\ix{w^p}$ is preserved if $\kappa=p-1$. Recall that $g=w^p$ obeys to the linear equation $g_t=\L g$. A straightforward computation (using the definition of $\mathcal L$ and Lemma \ref{lem:identies}) shows that
\begin{multline*}
\frac12\,\frac d{dt}\ix{\(|w'|^2\,\nu+\frac d{p-2}\,\big(|w|^2-\overline w^2\big)\)}\\
=-\ix{|w''|^2\,\nu^2}+2\,\frac{d-1}{d+2}\,\kappa\ix{w''\,\frac{|w'|^2}w\,\nu^2}-\frac d{d+2}\,\kappa\ix{\frac{|w'|^4}{w^2}\,\nu^2}\,,
\end{multline*}
since $\overline w=\(\ix{w^p}\)^{1/p}$ is independent of $t$. The r.h.s.~is negative if
\[
-\,\gamma_1=\(\frac{d-1}{d+2}\,\kappa\)^2-\frac d{d+2}\,\kappa\le0\,,
\]
that is, if $p\le2^\sharp:=\frac{2\,d^2+1}{(d-1)^2}$ when $d>1$, or $p>1$ when $d=1$, and this determines the expression \eqref{Eqn:gamma1} of $\gamma_1$. We have proved the following result.
%---------------------------------------------------------------------
\begin{proposition}\label{propflowbeta1}
For all $p\in [1, 2^\sharp]$ if $d>1$, $p>1$ if $d=1$, there exists a constant $\gamma_1>0$, such that if $w(t)$ is defined by \eqref{Eqn:Linear}, then
\[
\frac d{dt}\ix{w^{p}}=0\,,
\]
\[
-\frac d{dt}\ix{\(|w'|^2\,\nu+\frac d{p-2}\,\(w^2-\overline w^2\)\)}\ge2\,\gamma_1\ix{\frac{|w'|^4}{w^2}\,\nu^2}\,,
\]
where $\gamma_1$ is given by \eqref{Eqn:gamma1}.
\end{proposition}
%---------------------------------------------------------------------

%%%%%%%%%%%%%%%%%%%%%%%%%%%%%%%%%%%%%%%%%%%%%%%%%%%%%%%%%%%%%%%%%%%%%%
\subsection{The nonlinear case (\texorpdfstring{$\beta\ge1$)}{beta neq 1}}\label{Sec:betaNeq1} On $(-1,1)$, let us consider the flow defined by~\eqref{Eqn:Nonlinear}, that is, $w_t=w^{2-2\beta}\big(\L w+\kappa\,\frac{|w'|^2}w\big)$, and notice that
\[
\frac d{dt}\ix{w^{\beta p}}=\beta\,p\,(\kappa-\beta\,(p-2)-1)\ix{w^{\beta(p-2)}\,|w'|^2\,\nu}\,,
\]
so that $\overline w=\(\ix{w^{\beta p}}\)^{1/(\beta p)}$ is preserved if $\kappa=\beta\,(p-2)+1$. Similarly as in the previous case we calculate:
\begin{multline*}
-\frac1{2\,\beta^2}\,\frac d{dt}\ix{\(|(w^\beta)'|^2\,\nu+\frac d{p-2}\,\(w^{2\beta}-\overline w^{2\beta}\)\)}\\
=\ix{|w''|^2\,\nu^2}-2\,\frac{d-1}{d+2}\,(\kappa+\beta-1)\ix{w''\,\frac{|w'|^2}w\,\nu^2}\\
+\left[\kappa\,(\beta-1)+\,\frac d{d+2}\,(\kappa+\beta-1)\right]\ix{\frac{|w'|^4}{w^2}\,\nu^2}\,.
\end{multline*}
The r.h.s.~is negative if there exists a $\beta\in\R$ such that
\begin{multline*}
-\,\gamma=\(\frac{d-1}{d+2}\,(\kappa+\beta-1)\)^2-\left[\kappa\,(\beta-1)+\,\frac d{d+2}\,(\kappa+\beta-1)\right]\\
=\(\frac{d-1}{d+2}\,\beta\,(p-1)\)^2-\left[\frac d{d+2}\,\beta\,(p-1)+\big(1+\beta\,(p-2)\big)\,(\beta-1)\right]\le0\,,
\end{multline*}
\emph{i.e.},~$\gamma$ given by \eqref{Eqn:gamma} is nonnegative, and in that case we have found that
\be{energydecay}
-\frac1{2\,\beta^2}\,\frac d{dt}\ix{\(|(w^\beta)'|^2\,\nu+\frac d{p-2}\,\(w^{2\beta}-\overline w^{2\beta}\)\)}\ge\gamma\ix{\frac{|w'|^4}{w^2}\,\nu^2}\,.
\ee
This defines $\gamma$ as in~\eqref{Eqn:gamma}. Since the l.h.s.~of the inequality is quadratic in~$\beta$ and evaluates to $+1$ for $\beta=0$, a necessary and sufficient condition is that the discriminant, which amounts to
\[
\frac{4\,d\,(d-2)}{(d+2)^2}\,(p-1)\,(2^*-p)\quad\mbox{where}\quad 2^*:=\frac{2\,d}{d-2}
\]
takes nonnegative values, that is
\[
1\le p\le2^*\; \mbox{ if }\; d\ge 3\,.
\]
In dimension $d=2$ and $1$, the discriminant is respectively $2\,(p-1)$ and $\frac49\,(p-1)(p+2)$ and takes nonnegative values for any $p\ge1$ (we always assume that $p\ge1$). Altogether, we have proved the following result.
%---------------------------------------------------------------------
\begin{proposition}\label{propflowbetanot1}
For all $p\in [1, 2^*]$, there exist two constants, $\beta\in\R$ and  $\gamma>0$, such that if $w(t)$ is defined by \eqref{Eqn:Nonlinear}, then
\[
\frac d{dt}\ix{w^{\beta p}}=0\,,
\]
\[
-\frac1{2\beta^2}\,\frac d{dt}\ix{\(|(w^\beta)'|^2\,\nu+\frac d{p-2}\,\(w^{2\beta}-\overline w^{2\beta}\)\)}\geq \gamma\ix{\frac{|w'|^4}{w^2}\,\nu^2}\,,
\]
where $\gamma$ is given by \eqref{Eqn:gamma}.
\end{proposition}
%---------------------------------------------------------------------

Note that the flow described in Section~\ref{Sec:beta=1} is a special case of the flow defined by~\eqref{Eqn:Nonlinear} corresponding to $\beta=1$. We treat the cases $\beta=1$ and $\beta>1$ separately as they cover different ranges of $p$, namely $(1,2^\#)$ and $(2,2^*)$ respectively.

%%%%%%%%%%%%%%%%%%%%%%%%%%%%%%%%%%%%%%%%%%%%%%%%%%%%%%%%%%%%%%%%%%%%%%
%%%%%%%%%%%%%%%%%%%%%%%%%%%%%%%%%%%%%%%%%%%%%%%%%%%%%%%%%%%%%%%%%%%%%%
\section{Improved inequalities in the range \texorpdfstring{$p\in(2,2^*)$}{p in(2,2*)}}\label{Sec:Improved1}

In this section, we establish the expression of $\varphi_\beta$ that is used for stating Theorem~\ref{Thm:Main}. The computation corresponds to the one of J.~Demange in \cite{MR2381156} when $\beta=\frac4{6-p}$. In \cite{Demange-PhD}, J.~Demange noticed that there is a free parameter, which is equivalent to the parameter $\beta$ in our setting. Our purpose is to clarify the range of admissibility of $\beta$ and then optimize on it.

{}From here on, we assume that $\beta>1$. In Appendix~\ref{Sec:Appendix} we show that the set $\mathfrak B(p,d)$ of admissible $\beta$'s defined by \eqref{Eqn:admbeta} is non-empty if and only if one of the following conditions (Fig.~\ref{Fig2}) holds:
\[
(i)\;d\geq 3\,,\; p\in(2,2^*)\,,\quad (ii)\;d=2\,,\;p\in(2,9+4\sqrt3)\,,\quad(iii)\;d=1\,,\;p>2\,,
\]
and thus from now on we will take for granted that one of these conditions holds.
%*********************************************************************
\begin{figure}[ht]
\includegraphics[width=5cm]{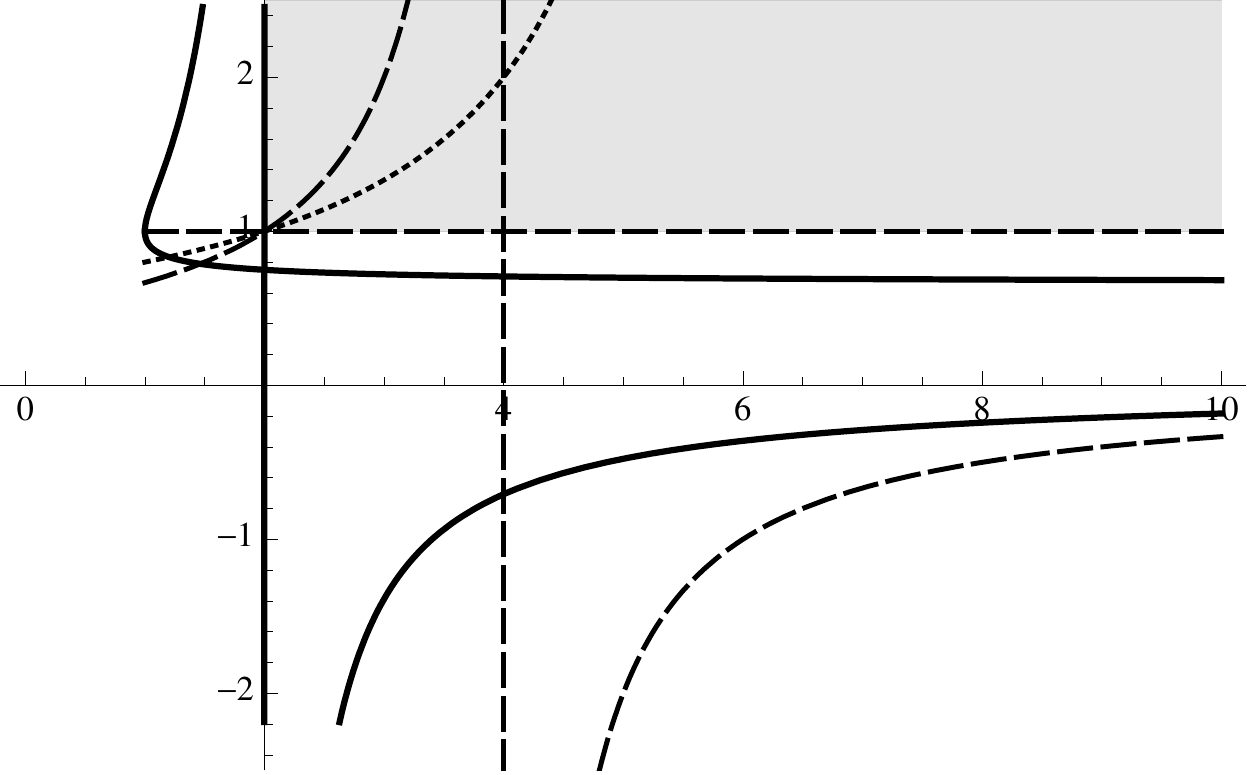}\hspace*{1cm}
\includegraphics[width=5cm]{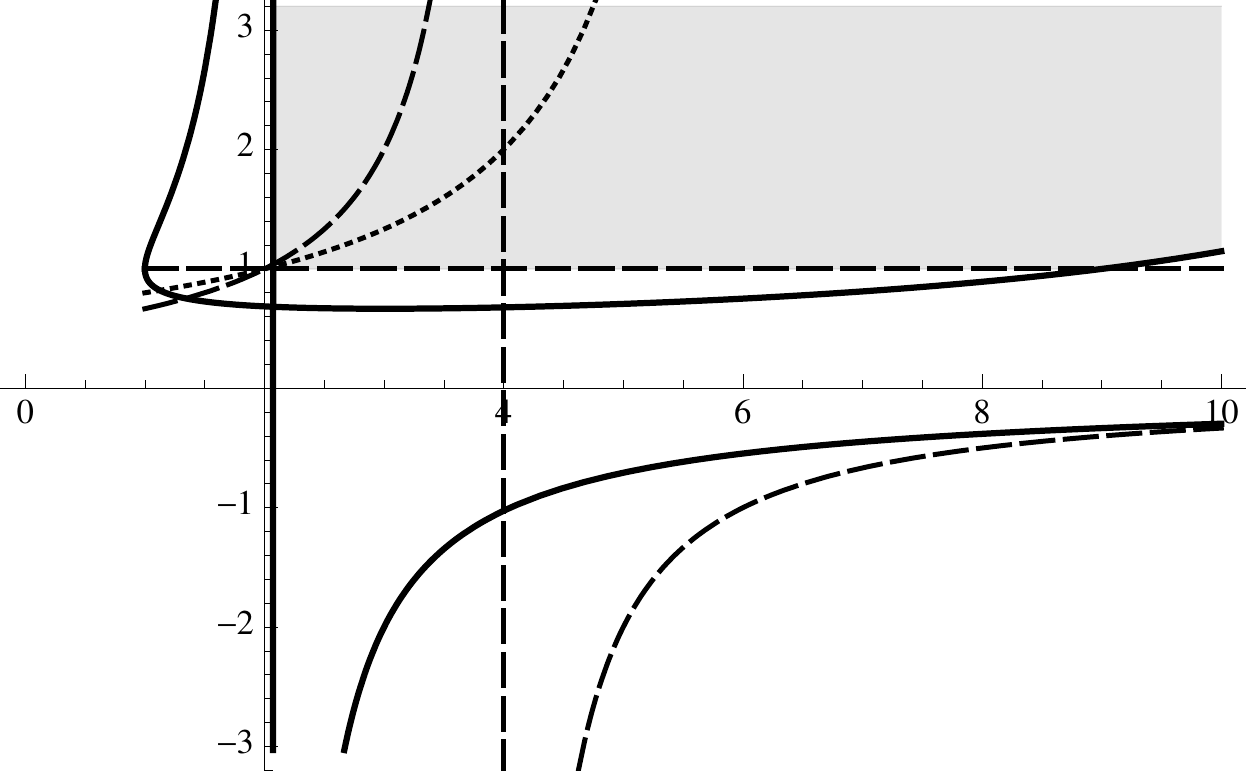}\\[0.3cm]
\includegraphics[width=5cm]{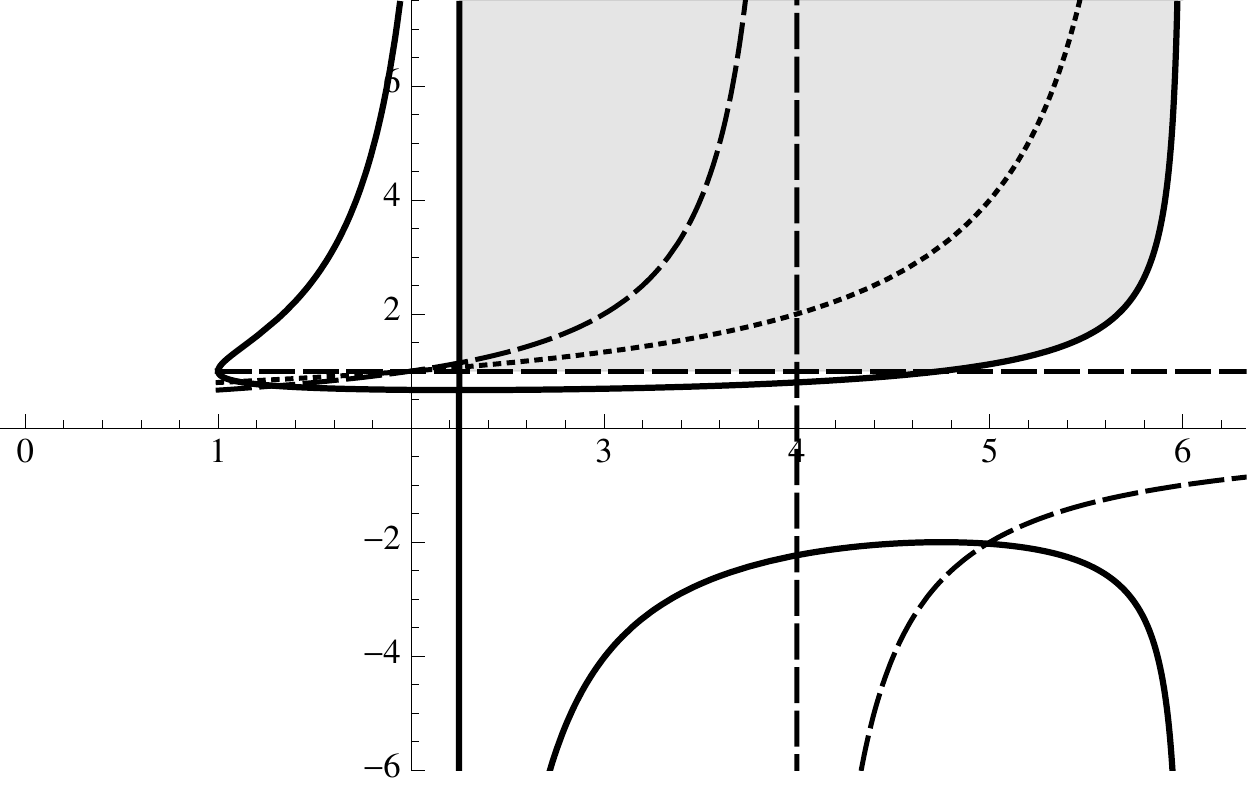}\hspace*{1cm}
\includegraphics[width=5cm]{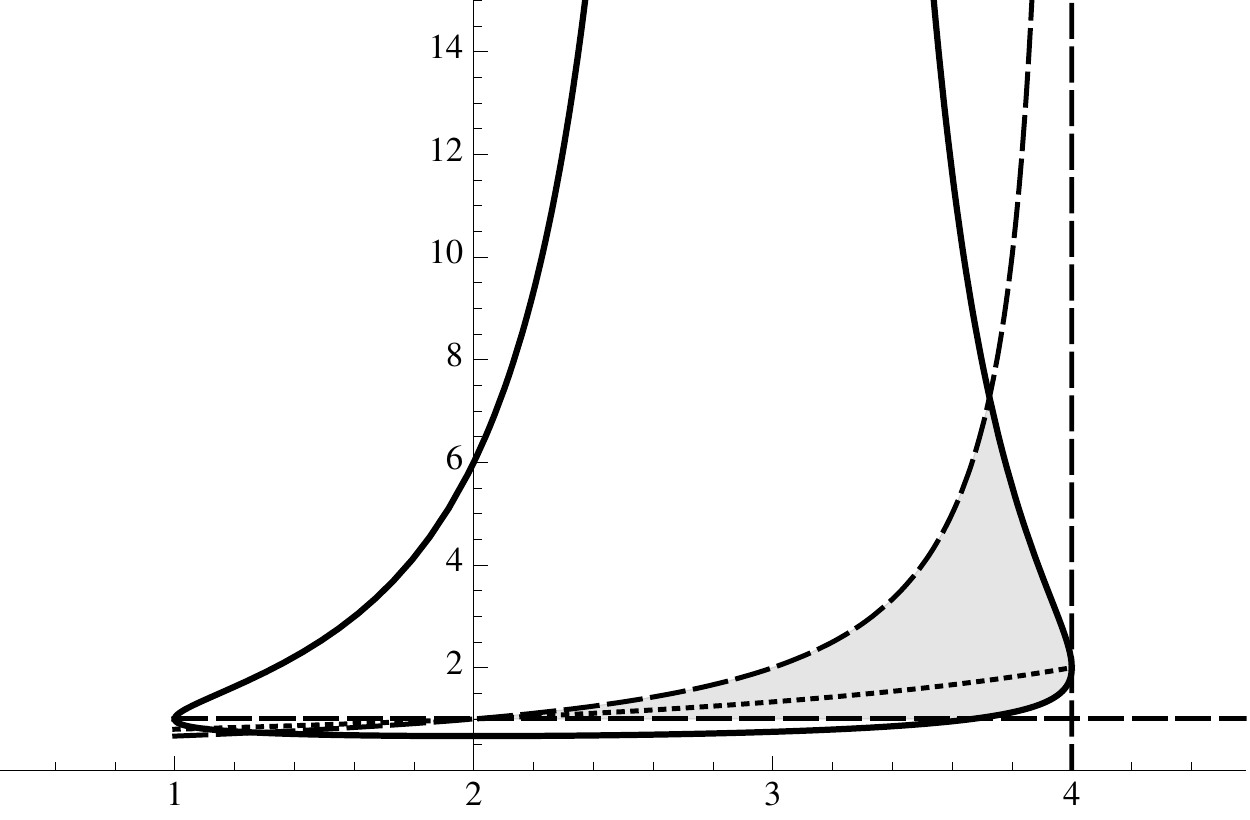}\\[0.3cm]
\includegraphics[width=5cm]{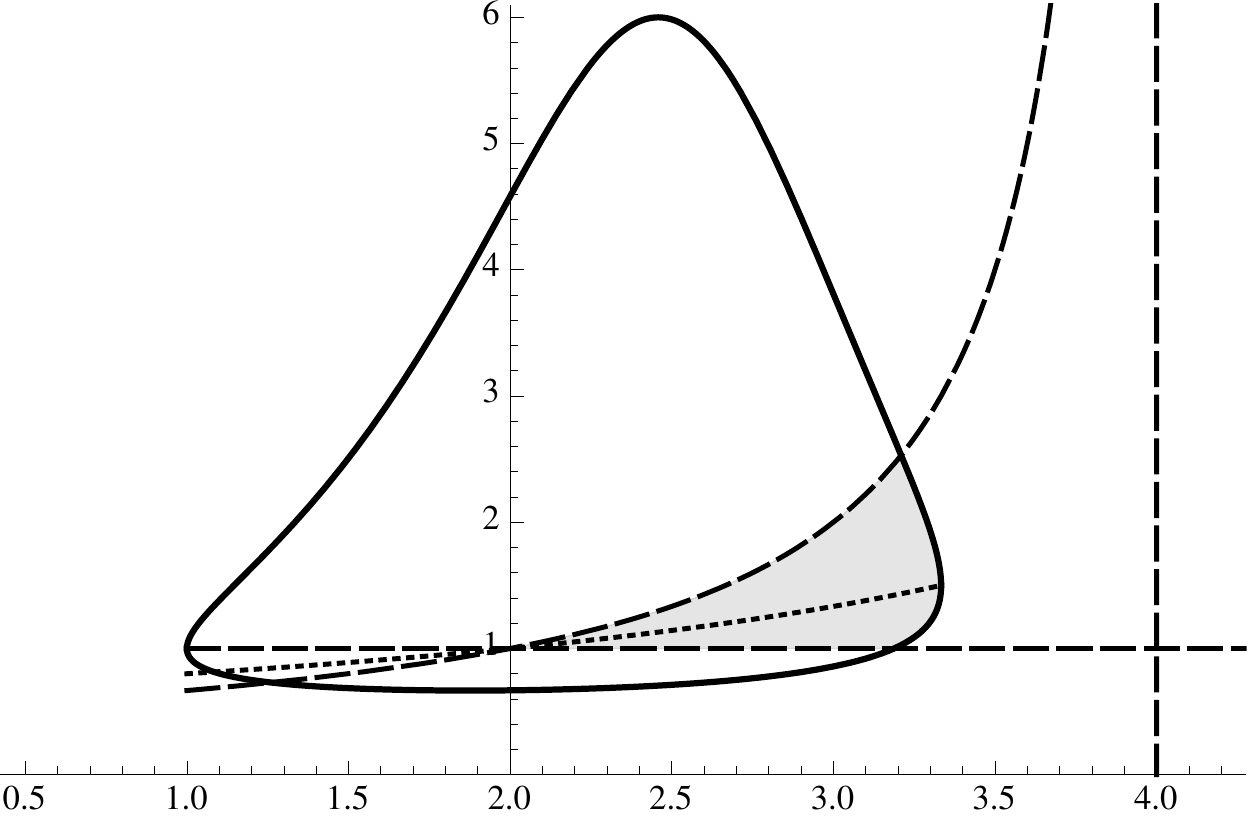}\hspace*{1cm}
\includegraphics[width=5cm]{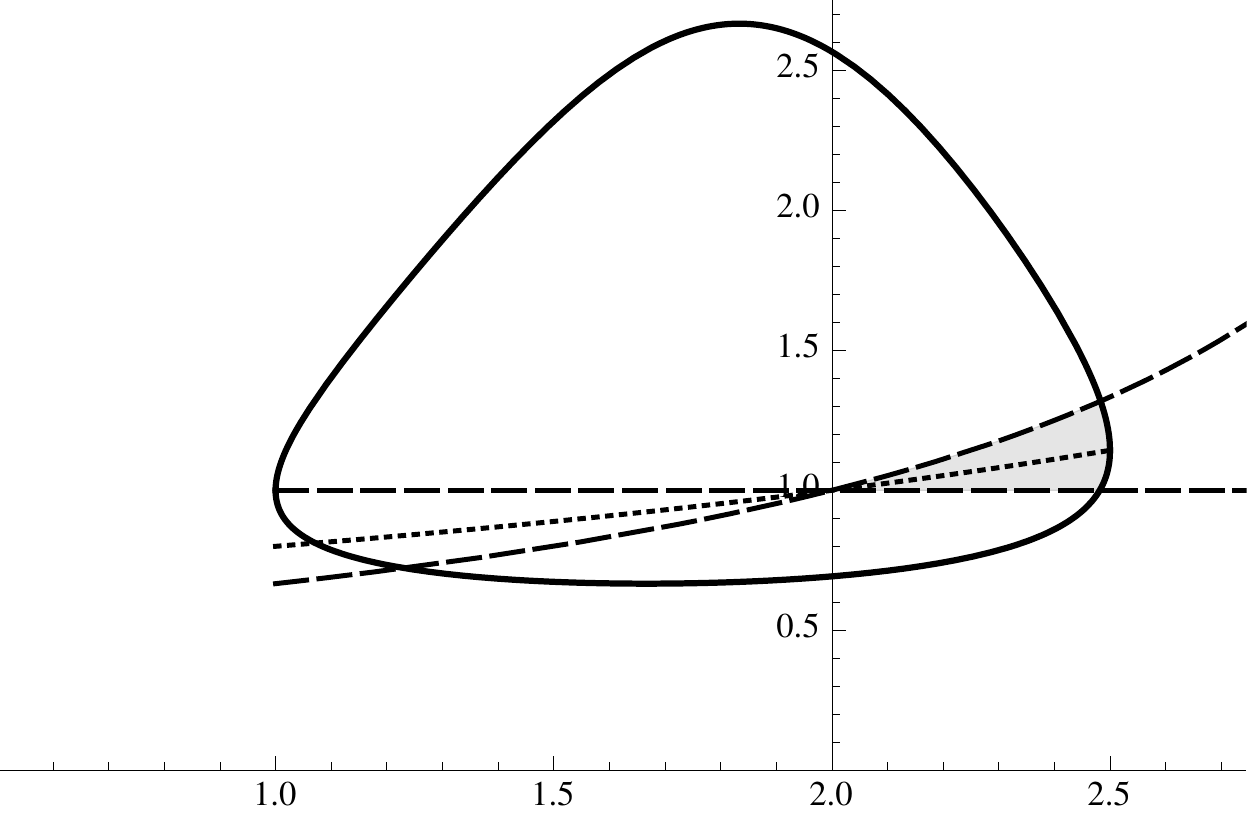}
\caption{\small\sl Generically $\gamma(\beta)=0$ has two solutions, $\beta_\pm(p,d)$. Plot of $p\mapsto\beta_\pm(p,d)$ with $d=1$ (left) and $d=2$ (right) on the first row, $d=3$ (left) and $d=4$ (right) on the second row, and $d=5$ (left) and $d=10$ (right) on the third row. The dotted line corresponds to $p\mapsto\frac4{6-p}$ and we have $\beta_-(p)\le\frac4{6-p}\le\beta_+(p)$ if and only if $\frac{18\,d}{17\,d-2}\le p\le2^*$. The dashed lines correspond to $\beta=1$ and $\beta=\frac2{4-p}$, and the interior of $\mathfrak B(p,d)$ corresponds to the grey area: see definition \eqref{Eqn:admbeta} of $\mathfrak B(p,d)$. Also see Fig.~\ref{Fig3-2} for more details in the  case $d=2$.\label{Fig2}}
\end{figure}
%*********************************************************************
In order to get the improved inequality, we make use of \eqref{energydecay} to get a lower bound for $\frac d{dt}(-\i+\,d\,\e)$. We note that the factor $\gamma$ is strictly positive for $\beta\in\mathfrak B(p,d)$. However, notice that there exist $\beta\notin\mathfrak B(p,d)$ such that $\gamma(\beta)>0$. Additional restrictions on the set of the admissible $\beta$'s indeed appear in what follows.

Let us show that we can get a lower bound for $\ix{\frac{|w'|^4}{w^2}\,\nu^2}$ by two different methods, using three points interpolation H\"older's inequalities.
\begin{enumerate}
\item With $\frac12+\frac{\beta-1}{2\beta}+\frac1{2\beta}=1$, H\"older's inequality shows that
\begin{multline*}
\ix{|w'|^2\,\nu}=\ix{\frac{|w'|^2}w\,\nu\cdot 1\cdot w}\\
\le\(\ix{\frac{|w'|^4}{w^2}\,\nu^2}\)^\frac12\(\ix 1\)^\frac{\beta-1}{2\beta}\(\ix{w^{2\beta}}\)^\frac1{2\beta}\,,
\end{multline*}
from which we deduce that
\[
\(\ix{\frac{|w'|^4}{w^2}\,\nu^2}\)^\frac12\ge\frac{\ix{|w'|^2\,\nu}}{\(\ix{w^{2\beta}}\)^\frac1{2\beta}}\,,
\]
because $d\nu_d$ is a probability measure.

\smallskip
\item With $\frac12+\frac{\beta-1}{\beta(p-2)}+\frac{2-(4-p)\beta}{2\beta(p-2)}=1$ and under the condition that
\be{Ineq:RangeBeta}
1<\beta\le\frac2{4-p}\quad\mbox{if}\quad p<4\,,
\ee
(but no condition if $p\ge4$, except $\beta>1$), H\"older's inequality shows that
\begin{multline*}
\frac1{\beta^2}\ix{|(w^\beta)'|^2\,\nu}\\
=\ix{w^{2(\beta-1)}\,|w'|^2\,\nu}=\ix{\frac{|w'|^2}w\,\nu\cdot w^\frac{p(\beta-1)}{p-2}\cdot w^{2\beta\delta}}\\
\le\(\ix{\frac{|w'|^4}{w^2}\,\nu^2}\)^\frac12\(\ix{w^{\beta p}}\)^\frac{\beta-1}{\beta(p-2)}\(\ix{w^{2\beta}}\)^\delta
\end{multline*}
with $\delta=\delta(\beta):=\frac{p-\,(4-p)\,\beta}{2\,\beta\,(p-2)}$ as in \eqref{Eqn:delta}, from which we deduce that
\[
\(\ix{\frac{|w'|^4}{w^2}\,\nu^2}\)^\frac12\ge\frac1{\beta^2}\,\frac{\ix{|(w^\beta)'|^2\,\nu}}{\(\ix{w^{2\beta}}\)^\delta}\,,
\]
because we have chosen $\ix{w^{\beta p}}=1$.
\end{enumerate}
By combining the two estimates we have proved the following.
%---------------------------------------------------------------------
\begin{lemma} Assume that \eqref{Ineq:RangeBeta} holds. For all $w\in\H^1\big((-1,1), d\nu_d\big)$, such that $\ix{w^{\beta p}}=1$,
\[
\ix{\frac{|w'|^4}{w^2}\,\nu^2}\ge\frac1{\beta^2}\,\frac{\ix{|(w^\beta)'|^2\,\nu}\ix{|w'|^2\,\nu}}{\(\ix{w^{2\beta}}\)^\delta}\,.
\]
\end{lemma}
%---------------------------------------------------------------------
%\medskip
\noindent Next, we can use the above estimates to show that there is a differential inequality relating $\e$ and $\i$. From the computations in Section~\ref{Sec:betaNeq1}, and Proposition \ref{propflowbetanot1}, we find that
\[
-\frac1{2\,\beta^2}\,\frac d{dt}\ix{\(|(w^\beta)'|^2\,\nu+\frac d{p-2}\,\(w^{2\beta}-\overline w^{2\beta}\)\)}\ge\gamma\ix{\frac{|w'|^4}{w^2}\,\nu^2}
\]
with $\overline w=\(\ix{w^p}\)^{1/p}=1$, hence
\[
-\,\i'+\,d\,\e'\ge2\,\gamma\,\frac{\ix{|(w^\beta)'|^2\,\nu}\ix{|w'|^2\,\nu}}{\(\ix{w^{2\beta}}\)^\delta}=-\frac\gamma{2\,\beta^2}\,\frac{\i\,\e'}{\(1\,-\,(p-2)\,\e\)^\delta}\,.
\]
Let us suppose that $\psi_\beta$ is a function that satisfies the ODE
\[
\frac{\psi_\beta''(\e)}{\psi_\beta'(\e)}=-\,\frac\gamma{2\,\beta^2}\,\(1\,-\,(p-2)\,\e\)^{-\delta}\,,
\]
and such that $\psi_\beta'>0$. It is elementary to show that such function exists when $\beta>1$. The l.h.s.~of the inequality
\[
\i'-\,d\,\e'-\frac\gamma{2\,\beta^2}\,\frac{\i\,\e'}{\(1\,-\,(p-2)\,\e\)^\delta}\le0
\]
can then be considered as a total derivative, namely
\[
\frac d{dt}\big(\i\,\psi_\beta'(\e)-d\,\psi_\beta(\e)\big)\le0\,,
\]
thus proving after an integration in $t\in[0,\infty)$ that
\[\label{Ineq:Improved}
d\,\frac{\psi_\beta(\e_0)}{\psi_\beta'(\e_0)}\le\i_0\,.
\]

%\medskip
 Next, we let
\[
\varphi_\beta(\e):=\frac{\psi_\beta(\e)}{\psi_\beta'(\e)}\,.
\]
It is then elementary to check that $\varphi_\beta$ satisfies the ODE
\[
\varphi_\beta'=1-\varphi_\beta\,\frac{\psi_\beta''(\e)}{\psi_\beta'(\e)}=1+\varphi_\beta\,\frac\gamma{2\,\beta^2}\,\(1\,-\,(p-2)\,\e\)^{-\delta}
\]
and $\varphi_\beta(0)=0$. Solving this linear ODE, we find that $\varphi_\beta$ is given by~\eqref{Eqn:varphibeta}. Notice that $\varphi_\beta$ is defined for any $\e\in[0,\tfrac1{p-2})$ and $\varphi_\beta(\e)>0$, for $\e>0$. From the equation satisfied by $\varphi_\beta$ we get that $\varphi'_\beta(\e)>1$ and $\varphi''_\beta(\e)>0$, $\e>0$, hence $\varphi_\beta(\e)>\e$ for any admissible $\beta$. As a consequence we also get that the functions $\varphi_\beta^{-1}$ and $\i\mapsto\i-\varphi_\beta^{-1}(\i)$ are increasing. Let us define (\emph{cf.} introduction)
\[
\varphi(\e):=\sup\Big\{\varphi_\beta(\e)\;:\;\beta\in\mathfrak B(p,d)\Big\}\,.
\]
Numerically we observe that $\varphi(\e)=\varphi_{\beta(\e)}(\e)$ for an optimal $\beta=\beta(\e)$, which explicitly depends on $\e$. For future reference we observe that $\varphi'(\e)>1$ and $\varphi''(\e)>0$, for a.e.~$\e>0$, hence $\varphi(\e)>\e$ and the functions $\varphi^{-1}$ and $\i\mapsto \i-\varphi^{-1}(\i)$ are increasing.

From the preceding considerations it follows that we have the inequality:
\[
d\,\varphi(\e)\le\i\,.
\]
Now we recall that this inequality is obtained under the assumption $\nrm up^2=1$, namely $\e=\frac1{p-2}\,\big(1-\nrm u2^2\big)$. Then, if we do not normalize the $\mathrm L^p$ norm we get in general:
\[
d\,\nrm up^2\,\varphi\!\(\frac{\e}{\nrm up^2}\)\le\i\,.
\]
Now, remembering the definition of the improved inequality, we need to find the relation between the l.h.s.~of this estimate, which is a function of $\e/\nrm up^2$, and the function $d\,\nrm u2^2\,\Phi\big(\e/\nrm u2^2\big)$. This is quite easy since we have
\[
\frac{\nrm up^2}{\nrm u2^2}=1+(p-2)\,s\quad\mbox{with}\quad s=\frac{\e}{\nrm u2^2}\,.
\]
By straightforward manipulations we get from this
\[
\nrm up^2\,\varphi\!\(\frac{\e}{\nrm up^2}\)=\big(1+(p-2)\,s\big)\,\varphi\!\(\frac s{1+(p-2)\,s}\)=:\Phi(s)
\]
where $\Phi$ is the function under consideration in Theorem \ref{Thm:Main}. Based on the properties of $\varphi$, it is easy to check that
$\Phi(0)=0$, $\Phi'(0)=1$ and $\Phi''(s)>0$. In Fig.~\ref{Fig4} we show the measure of improvement between the improved inequality and the standard inequality.

This ends the proof of Theorem \ref{Thm:Main} in the case $p\in(2,2^*)$ if $d=1$ or $d\ge 3$, $p\in(2,9+4\sqrt3)$, if $ d=2$.\finprf
%*********************************************************************
\begin{figure}[ht]
\includegraphics[width=5cm]{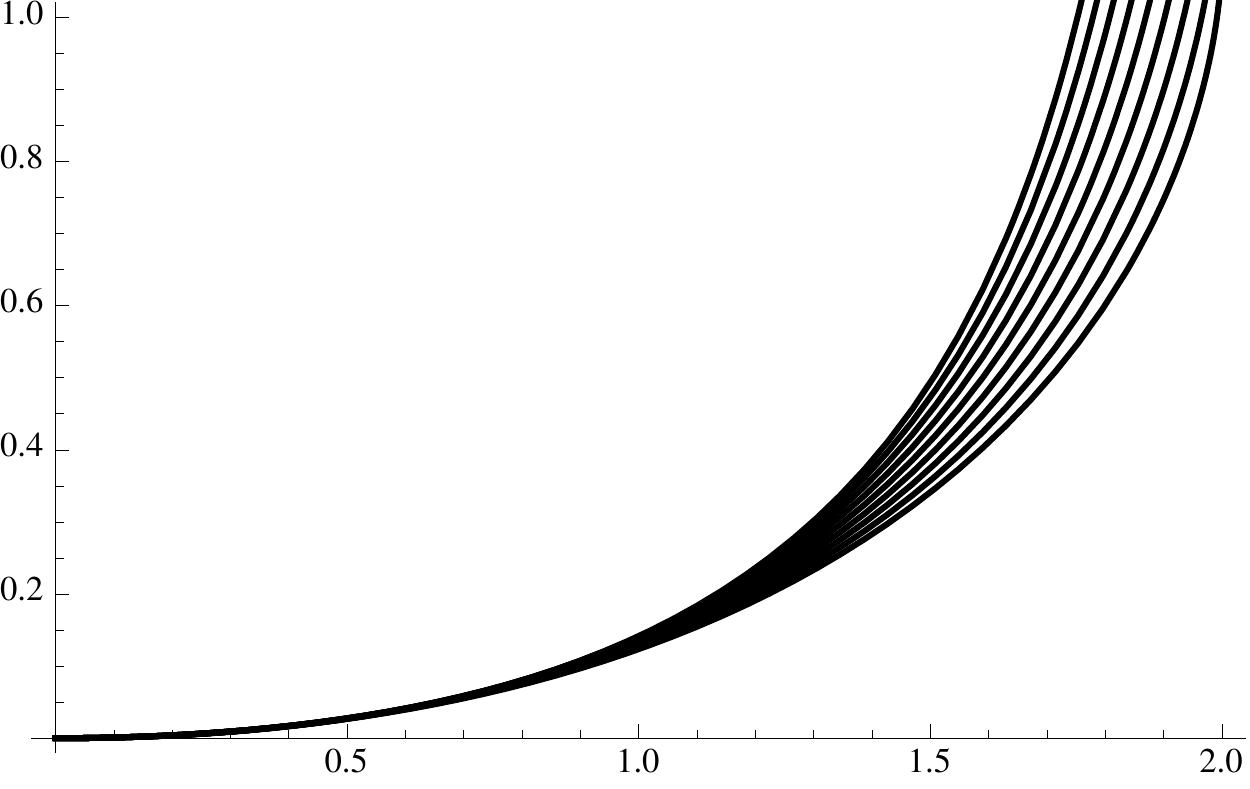}\hspace*{30pt}\includegraphics[width=6cm]{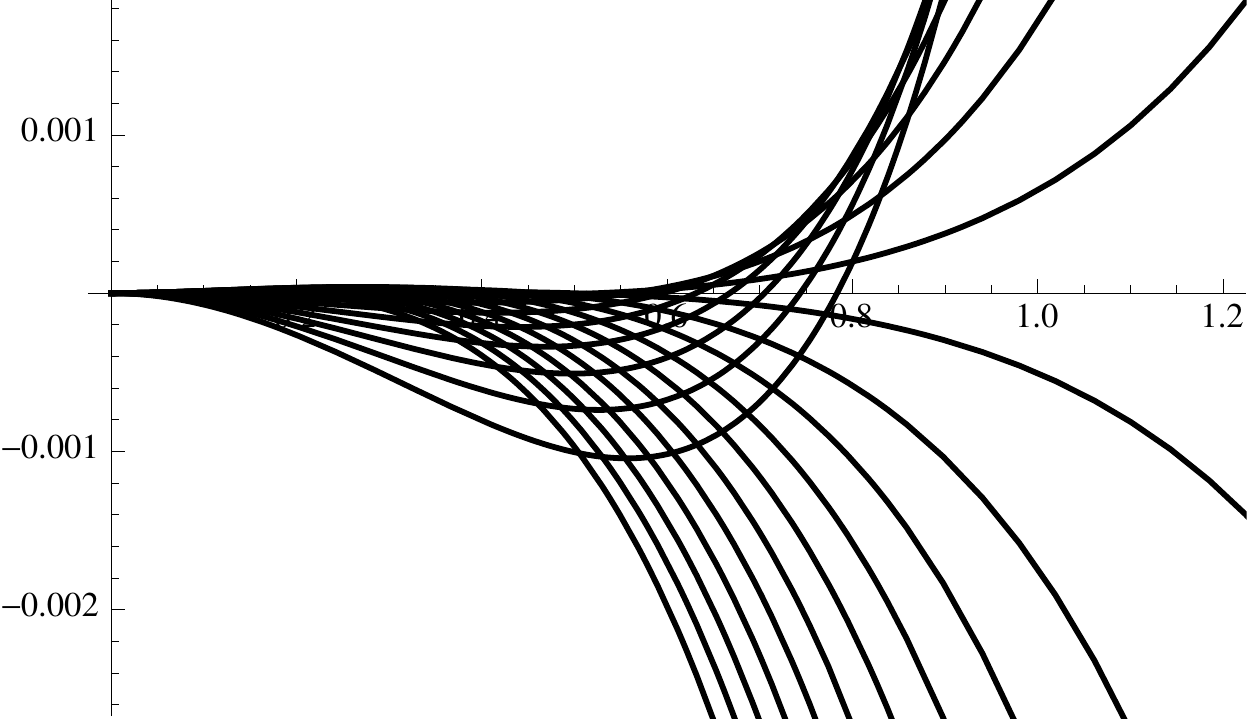}
\caption{\small\sl\label{Fig4} Plot of the function $\varphi_\beta$ corresponding to Theorem~\ref{Thm:Main}. Here $d=5$, $p=2.5$ and $\beta\approx2.383$, $2.267$, $2.15$, $2.033$, $1.917$, $1.8$, $1.683$, $1.567$, $1.45$. Left: plot of $\e\mapsto\varphi_\beta(\e)-\e$ for $\e\in(0,1/(p-2))$. Right: plot of $\e\mapsto\varphi_\beta(\e)-\varphi_{\beta_0}(\e)$ where $\beta_0=4/(6-p)$ corresponds to the choice made in \cite{MR2381156}.}
\end{figure}
%*********************************************************************
As a conclusion for this section, let us comment on the literature and emphasize the new results. In \cite{MR2381156} J.~Demange gives a proof by considering a flow which corresponds to \eqref{Eqn:Nonlinear}, in the special case $\beta=\frac4{6-p}$. Here we generalize it to a larger but explicit range of $\beta$'s, as was done in \cite[pages 122--130]{Demange-PhD} (see in particular Proposition 3.12.5). Moreover we explicitly show how to optimize the interpolation inequalities with respect to $\beta$ and we specify the range of admissible $\beta$'s. See Appendix \ref{Sec:Appendix} for more details.

%%%%%%%%%%%%%%%%%%%%%%%%%%%%%%%%%%%%%%%%%%%%%%%%%%%%%%%%%%%%%%%%%%%%%%
%%%%%%%%%%%%%%%%%%%%%%%%%%%%%%%%%%%%%%%%%%%%%%%%%%%%%%%%%%%%%%%%%%%%%%
\section{Improved inequalities in the range \texorpdfstring{$p\in(1,2^\#)$}{p in(1,2sharp)}: Proof of Theorem \ref{Thm:Main} when \texorpdfstring{$p\in(1,2)$ and when $p=2$}{p in(1,2) and when p=2}}\label{Sec:Improved2}

In this section we adapt the Bakry-Emery approach (which amounts to write a differential inequality for $\i=-\,\e'$) and improve it by taking into account the remainder term as in \cite{MR2152502}. Here we assume that $\beta=1$. Our result is primarily an improvement of the existing results for $p\in(1,2)$, but we work in the larger range $p\in(1,2^\sharp)$ (see Section~\ref{Sec:beta=1}). We will see that the limitation on the exponent appears naturally since we do not allow any freedom for~	$\beta$. This special exponent $2^\sharp$ was already noticed in \cite{MR808640,DEKL2012}. In the range $p>2$, the limitation $p\le2^\#$ is equivalent to require that $\beta=1\in\mathfrak B(p,d)$. When $p\in(2,2^\#]$, the computations of this section are a limiting case of those in Section~\ref{Sec:Improved1}. Since estimates have to be adapted and since the range in $p$ is anyway different, we handle this case separately.

Assuming that $\nrmx up=1$ and $u$ is a solution of~\eqref{Eqn:Linear} and following the calculations of Section~\ref{Sec:beta=1} and Section~\ref{Sec:Improved1}, one can check the validity of the differential inequality
\[
\e''+\,d\,\e'\ge \frac{\gamma_1}{2}\,\frac{|\e'|^2}{1-\,(p-2)\,\e}\,,
\]
where $\gamma_1>0$ has been defined in Section~\ref{Sec:beta=1}. The estimate is a simple consequence of the Cauchy-Schwarz inequality $|\e'|^2=\nrmx{w'}2^2\le\nrmx w2\,\nrmx{(w')^2\!/w}2$. We observe that we have the boundary conditions $\e(t=0)=\e_0$, $\i(t=0)=\i_0$ and
\[
\lim_{t\to\infty}\e(t)=\lim_{t\to\infty}\i(t)=0\,.
\]
%---------------------------------------------------------------------
\begin{proposition}\label{Lem:ODE} With the above notations, and assuming $p\in[1,2)\cup(2,2^\#]$ and $\ix{w^{ p}}=1$, we have the following estimate
\[
d\,\varphi_1(\e)\le\i\,,
\]
where $\varphi_1$ is defined by \eqref{Eqn:varphi1} and such that $\varphi_1(0)=0$, $\varphi_1'(0)=1$ and $\varphi_1''(s)>0$ for any $s>0$ if $p>2$, or for any $s\in(0,1/(2-p))$ if $p\in(1,2)$\end{proposition}
%---------------------------------------------------------------------

%*********************************************************************
\begin{figure}[ht]
\includegraphics[width=5cm]{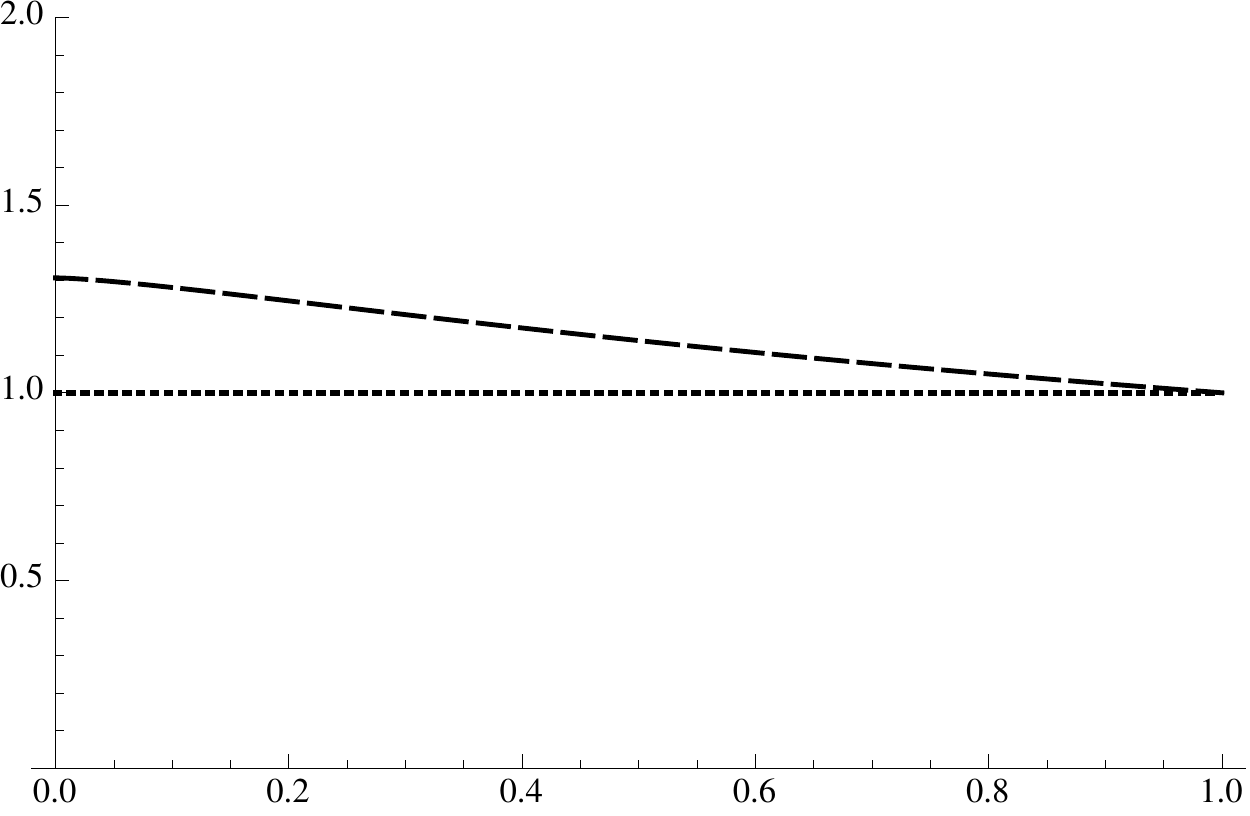}\hspace*{30pt}\includegraphics[width=5cm]{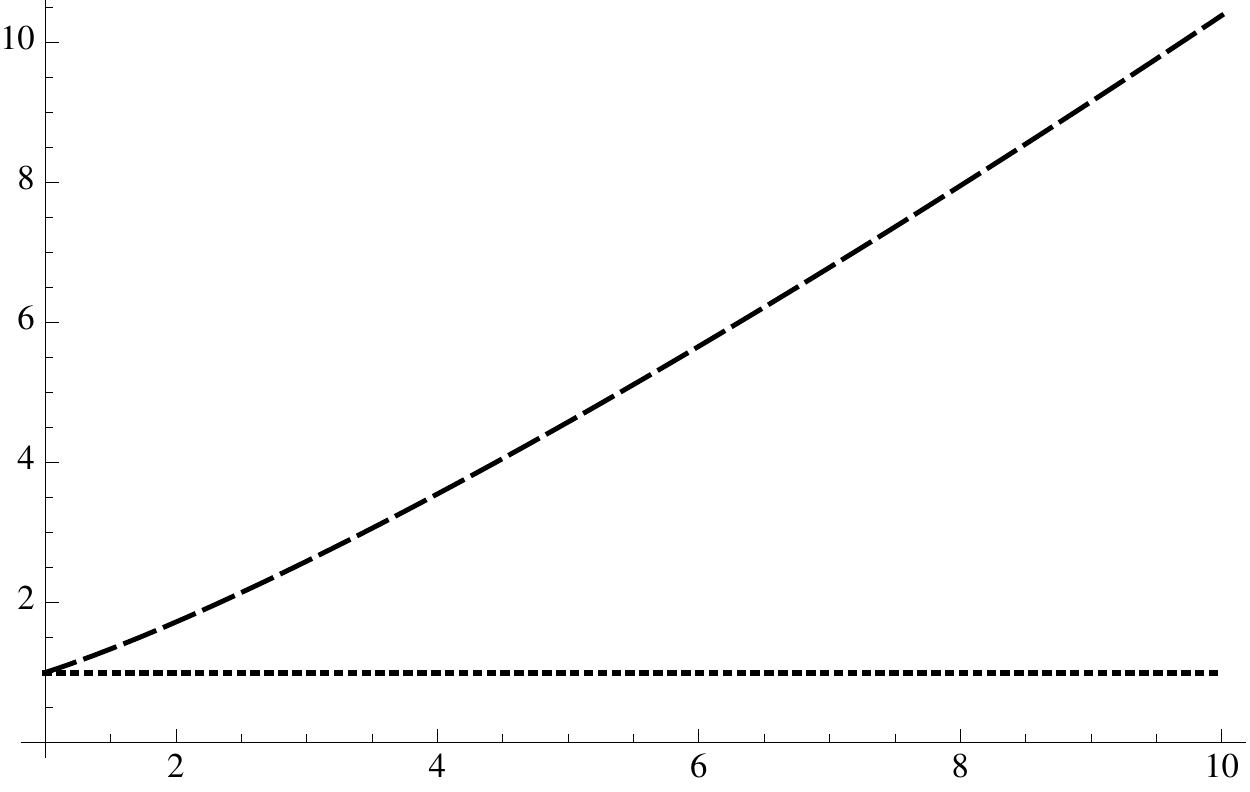}
\caption{\small\sl\label{Fig5-1} The plot of $\textstyle \xi\mapsto\frac{2\,(2-p)}{2\,(2-p)-\gamma_1}\,\frac{\xi^{2+\frac{\gamma_1}{p-2}}-1}{\xi^2-1}$ represents the improvement achieved in Proposition~\ref{Lem:ODE} compared to the standard form of the inequality, which is represented by $1$.
Left: $p\in(1,2)$ and $\xi=\nrm fp/\nrm f2$ is in the range $(0,1)$. The plot corresponds to $p=3/2$ and $d=2$. Right: $p\in(2,2^\#)$ and $\xi=\nrm fp/\nrm f2$ is in the range $\xi>1$. The plot corresponds to $p=5/2$ and $d=2$, $2^\#=9$.}
\end{figure}
%*********************************************************************
\begin{proof} Let us define $\mathsf h(t):=1-\,(p-2)\,\e(t)$. When $p\in(2,2^\#)$ we have
\[
\mathsf h'=-\,(p-2)\,\e' >0\,,\quad\mathsf h''=-\,(p-2)\,\e''<0\,,
\]
while when $p\in (1,2)$ these inequalities are changed to their opposities. Also $\mathsf h(t)\in (0,1)$, $\mathsf h(+\infty)=1$, $\mathsf h'(+\infty)=0$. Our differential inequality takes form:
\[
-\,\frac1{p-2}\,(\mathsf h''+\,d\,\mathsf h')-\frac{\gamma_1}{2\,(p-2)^2}\frac{|\mathsf h'|^2}{\mathsf h}\geq 0\,,
\]
which upon multiplying by $\frac{p-2}{\mathsf h'}$ and rearranging leads to

\[
\frac d{dt}\log\mathsf h'+\frac{\gamma_1}{2\,(p-2)}\,\frac d{dt}\log\mathsf h \quad\begin{cases}\le -\,d \;\mbox{ if }\; p>2\,,\\
\geq -\,d \;\mbox{ if }\; p<2\,,
\end{cases}
\]
and after one integration from $0$ to $t$ yields
\[
\mathsf h'(t)\,\mathsf h(t)^\frac{\gamma_1}{2\,(p-2)}\quad\begin{cases} \le e^{-dt}\,\mathsf h'(0)\,\mathsf h(0)^\frac{\gamma_1}{2\,(p-2)} \;\mbox{ if }\; p>2\,,
\\
\geq e^{-dt}\,\mathsf h'(0)\,\mathsf h(0)^\frac{\gamma_1}{2\,(p-2)} \;\mbox{ if }\; p>2\,. \end{cases}
\]
Integrating now from $0$ to $\infty$ we find
\[
\frac{2\,(p-2)}{\gamma_1+2\,(p-2)}\Big[1-\mathsf h(0)^{\frac{\gamma_1}{2\,(p-2)}+1}\Big]\le\frac1d\,\mathsf h'(0)\,\mathsf h(0)^\frac{\gamma_1}{2\,(p-2)}\,,
\]
from which it follows that
\[
\frac{2d\,(p-2)}{\gamma_1+2\,(p-2)}\Big
[\mathsf h(0)^{-\gamma_1/2\,(p-2)}-\mathsf h(0)\Big
]\le\mathsf h'(0)\,.
\]
This is the desired inequality since $\mathsf h'(0)=(p-2)\,\i_0$. This ends the proof of Proposition~\ref{Lem:ODE}, as a particular case, and completes the proof of Theorem \ref{Thm:Main} in the case $p\in (1,2)$.
\end{proof}

An interesting consequence is that the limit case $p=2$ gives an improvement of the logarithmic Sobolev inequality, which is stated in Theorem \ref{Thm:Main} for $p=2$. We check that as $p\to2$, $\gamma_1$ converges to $\gamma_1^*=(4\,d-1)/(d+2)^2$. The conclusion holds by observing that~$\e$ converges to $\frac 12\,\ix{w^2\,\log\big(|w|^2/\nrmx w2^2\big)}$, while
\[
\lim_{p\to2}\frac {2d}{\gamma_1+2\,(p-2)}\left[\(1-\,(p-2)\,\e_0\)^{-\frac{\gamma_1}{2\,(p-2)}}-1+\,(p-2)\,\e_0\right]=\frac {2d}{\gamma_1^*}\left[e^{\gamma_1^*\,\e_0/2}-1\right]\,.
\]

%%%%%%%%%%%%%%%%%%%%%%%%%%%%%%%%%%%%%%%%%%%%%%%%%%%%%%%%%%%%%%%%%%%%%%
%%%%%%%%%%%%%%%%%%%%%%%%%%%%%%%%%%%%%%%%%%%%%%%%%%%%%%%%%%%%%%%%%%%%%%
\section{Improved inequalities based on spectral estimates}\label{Sec:Spectral}

In this section, we first adapt the results and the proofs of \cite{ABD} to the setting of the ultraspherical operator. The method is the same, but gives a point of view which complements the other improved inequalities of this paper in the range $p\in(1,2)$. For completeness, we shall give short proofs and refer to \cite{ABD} for more details and references in the framework of probability measures. We shall next review some results in the critical case $p=2^*$ when $d>2$ or its counterpart when $d\le2$ (Section~\ref{Spectral:critical}), in order to raise an open question (Section~\ref{Sec:Open}).

%%%%%%%%%%%%%%%%%%%%%%%%%%%%%%%%%%%%%%%%%%%%%%%%%%%%%%%%%%%%%%%%%%%%%%
\subsection{Improved inequalities for reduced classes of functions in the range \texorpdfstring{$p\in(1,2)$}{p in(2,2)}}\label{Sec:Improved3}

We will now establish Theorem \ref{Thm:Main2} by proving a series of intermediate results. We recall that  $E_j$ denotes the eigenspace associated with the eigenvalue $\lambda_j$ of the Laplace-Beltrami operator $\Lap$ on the sphere (see the introduction for more details).
%---------------------------------------------------------------------
\begin{proposition}\label{Prop:Spectral} Let $k\ge1$ be an integer. Assume that $u\in\mathrm L^2(\S^d,d\mu)$ is such that \eqref{Eqn:Orthogonality} holds. Then the improved inequality
\be{Eqn:BecknerExtended2}
\iS{|\nabla u|^2}\ge\frac{d\,\alpha_k}{1-(p-1)^{\alpha_k}}\left[\iS{|u|^2}-\(\iS{|u|^\qp}\)^{2/\qp}\right]
\ee
holds for any $\qp\in[1,2)$.\end{proposition}
%---------------------------------------------------------------------

\begin{proof} To establish the inequality, we proceed in two steps.

\medskip
\noindent\textbf{$1^{\rm st}$ step: Nelson's hypercontractivity result.\/} The method is exactly the same as in \cite[Proposition 5]{DEKL2012}. Although the result can be established by direct methods, we follow here the strategy of Gross in \cite{Gross75}, which proves the equivalence of the optimal hypercontractivity result and the optimal logarithmic Sobolev inequality.

Consider the heat equation of $\S^d$, namely
\[
\frac{\partial f}{\partial t}=\Lap f\,,
\]
with initial datum $f(t=0,\cdot)=u\in\mathrm L^\qp(\S^d)$, for some $\qp\in(1,2)$, and let $F(t):=\nrm {f(t,\cdot)}{\lambda(t)}$. The key computation goes as follows.
\begin{multline*}
\frac{F'}F=\frac d{dt}\,\log F(t)=\frac d{dt}\,\left[\frac 1{\lambda(t)}\,\log\(\iS{|f(t,\cdot)|^{\lambda(t)}}\)\right]\\
=\frac{\lambda'}{\lambda^2\,F^\lambda}\left[\iS{v^2\log\(\frac{v^2}{\iS{v^2}}\)}-4\,\frac{\lambda-1}{\lambda'}\,\iS{|\nabla v|^2}\right]
\end{multline*}
with $v:=|f|^{\lambda(t)/2}$. Assuming that $4\,\frac{\lambda-1}{\lambda'}=\frac 2d$ so that $F'\le0$ by the logarithmic Sobolev inequality~\eqref{Ineq:LogSobolev}, that is
\[
\frac{\lambda'}{\lambda-1}=2\,d\,,
\]
we find that
\[
\log\(\frac{\lambda(t)-1}{\qp-1}\)=2\,d\,t
\]
if we require that $\lambda(0)=\qp<2$. Let $t_*>0$ be such that $\lambda(t_*)=2$, \emph{i.e.}
\[
t_*=-\,\frac{\log(p-1)}{2\,d}\,.
\]
As a consequence of the above computation, we have
\be{Ineq:Nelson}
\nrm{f(t_*,\cdot)}2\le\nrm u{\qp}\quad\mbox{because}\quad\frac 1{\qp-1}=e^{2\,d\,t_*}\,.
\ee

\medskip
\noindent\textbf{$2^{\rm nd}$ step: Spectral decomposition.\/} Let $u=\bar u+\sum_{j>k}f_j$ be a decomposition of the initial datum on the eigenspaces of $-\Lap$ so that $-\Lap f_j=\lambda_j\,f_j$. Let $a_j=\nrm{f_j}2^2$, for any $j>0$, and $a_0=\bar u^2$. As a straightforward consequence of this decomposition, we know that $\nrm u2^2=a_0+\sum_{j>k}a_j$, $\nrm{\nabla u}2^2=\sum_{j>k}\lambda_j\,a_j$,
\[
\nrm{f(t_*,\cdot)}2^2=a_0+\sum_{j>k}a_j\,e^{-2\,\lambda_j\,t_*}\,.
\]
Using \eqref{Ineq:Nelson}, it follows that
\begin{multline*}
\frac{\iS{|u|^2}-\(\iS{|u|^\qp}\)^{2/\qp}}{2-\qp}\le\frac{\(\iS{|u|^2}\)-\iS{|f(t_*,\cdot)|^2}}{2-\qp}\\
=\frac 1{2-\qp}\sum_{j>k}\lambda_j\,a_j\,\frac{1-e^{-2\,\lambda_j\,t_*}}{\lambda_j}\,.
\end{multline*}
Since $\lambda\mapsto\frac{1-e^{-2\,\lambda\,t_*}}{\lambda}$ is decreasing, we can bound $\frac{1-e^{-2\,\lambda_j\,t_*}}{\lambda_j}$ by $\frac{1-e^{-2\,\lambda_{k+1}\,t_*}}{\lambda_{k+1}}$ for any $k\ge 1$. This proves that
\begin{multline*}
\frac{\iS{|u|^2}-\(\iS{|u|^\qp}\)^{2/\qp}}{2-\qp}\le\frac{1-e^{-2\,\lambda_{k+1}\,t_*}}{(2-\qp)\,\lambda_{k+1}}\sum_{j>k}\lambda_j\,a_j\\
=\frac{1-e^{-2\,\lambda_{k+1}\,t_*}}{(2-\qp)\,\lambda_{k+1}}\,\nrm{\nabla u}2^2\,.
\end{multline*}
The conclusion easily follows.
\end{proof}

Using the same spectral method as in the proof of Proposition~\ref{Prop:Spectral}, as in \cite{ABD}, we will next a establish a more general interpolation inequality.
%---------------------------------------------------------------------
\begin{proposition}\label{Prop:Spectral2} Let $k\ge1$ be an integer. Assume that $u\in\mathrm L^2(\S^d,d\mu)$ is such that
\[
\iS{u\,e}=0\quad\forall\,e\in E_j\,,\quad j=1\,,2\ldots k\,.
\]
Then, with the same notations as in Proposition~\ref{Prop:Spectral}, the improved inequality
\be{Ineq:gammaRSI}
\nrm{\nabla u}2^2\ge\frac{d\,\alpha_k}{1-(p-1)^{{\alpha_k\gamma}/2}}\left[\nrm u2^2-\nrm up^\gamma\,\nrm u2^{2-\gamma}\right]
\ee
holds for any $\qp\in[1,2)$ and any $\gamma\in(0,2)$.\end{proposition}
%---------------------------------------------------------------------
\begin{proof} The computations are analogous to the ones of Proposition~\ref{Prop:Spectral}. If $f$ is a solution to the heat equation with initial datum $u$, then
\begin{multline*}
\nrm u2^2-\nrm{f(t_*,\cdot)}2^\gamma\nrm u2^{2-\gamma}\\
=a_0+ \sum_{j>k} a_j-\left(a_0+\sum_{j>k} a_j\, e^{-2\lambda_jt}\right)^{\frac\gamma 2}\left(a_0+\sum_{j>k} a_j\right)^{\frac{2-\gamma}2}
\end{multline*}
can be estimated using H\"older's inequality by
\begin{multline*}
a_0+\sum_{j>k}a_j\,e^{-\gamma\lambda_jt_*}=a_0+\sum_{j>k} \left(a_j\,e^{-2\lambda_jt_*}\right)^{\frac\gamma 2}\cdot\, a_j^{\frac{2-\gamma}2}\\
\le\left(a_0+\sum_{j>k} a_j\, e^{-2\lambda_jt_*}\right)^{\frac\gamma 2}\left(a_0+\sum_{j>k} a_j\right)^{\frac{2-\gamma}2}.
\end{multline*}
Nelson's estimate~\eqref{Ineq:Nelson} shows that
\[
\nrm u2^2-\nrm up^\gamma\nrm u2^{2-\gamma}\le\sum_{j>k} a_j\,\left(1-e^{-\gamma\lambda_jt_*} \right)\le \frac {1-e^{-\gamma\lambda_{k+1}t_*}}{\lambda_{k+1}}\sum_{j>k}\lambda_j\,a_j
\]
using the decay of $\lambda\mapsto\left(1-e^{-\gamma\lambda t_*}\right)/\lambda$ and conclude as before with $\sum_{j>k}\lambda_j\,a_j=\iS{|\nabla u|^2}$.\end{proof}

An optimization on $\gamma$ can also be done, as in \cite{ABD}.
%---------------------------------------------------------------------
\begin{corollary}\label{Cor:Spectral3} Let $k\ge1$ be an integer and $\qp\in(1,2)$. Assume that $u\in\mathrm L^2(\S^d,d\mu)$ is such that
\[
\iS{u\,e}=0\quad\forall\,e\in E_j\,,\quad j=1\,,2\ldots k\,.
\]
Then the following estimate holds
\be{Ineq:gammaRSI2}
\iS{|\nabla u|^2}\ge\frac{2\,d}{\log(p-1)}\,\nrm u2^2\,\log\(\frac{\nrm up}{\nrm u2}\)
\ee
if
\[
\nrm up\le(p-1)^{\alpha_k/2}\,\nrm u2\,.
\]\end{corollary}
%---------------------------------------------------------------------
\begin{proof} We optimize the r.h.s.~of Inequality \eqref{Ineq:gammaRSI} w.r.t.~$\gamma\in (0,2)$, that is, we maximize
\[
\gamma\mapsto d\,\alpha_k\,\nrm u2^2\frac{\nrm u2^2-\nrm up^\gamma\nrm u2^{2-\gamma}}{[1-(p-1)^{{\alpha_k\gamma}/2}]\nrm u2^2}=d\,\alpha_k\,\nrm u2^2\,h(\gamma)\,,
\]
with
\[
h(\gamma):=\frac{1-a^\gamma}{1-b^\gamma}\;,\quad a=\frac{\nrm fp}{\nrm f2}\le1\quad\mbox{and}\quad b=(p-1)^{\alpha_k/2}\le1\,.
\]
We write $h(\gamma)=g(b^\gamma)$ with $g(y):=(1-y^\frac{\log a}{\log b})/(1-y)$. For $a<b<1$ the function $g(y)$ is monotone increasing because $y\mapsto y^{\log a/\log b}$ is convex. Hence, $h(\gamma)$ is monotone decreasing. Analogously, $h$ is monotone increasing for $b<a<1$. Hence, the maximum of the function $h(\gamma)$ on $[0,2]$ is either $h(2)$ (if $a>b$) or $\lim_{\gamma\to0}h(\gamma)$ (in the case $a<b$). This yields the conclusion.\end{proof}

Summarizing, Inequalities \eqref{Eqn:BecknerExtended2}, \eqref{Ineq:gammaRSI} and \eqref{Ineq:gammaRSI2} can be written respectively with $i=1$, $2$, and $3$ as
\[
\iS{|\nabla u|^2}\ge d\,\,\nrm u2^2\,\chi_k^{(i)}\(\frac{\nrm up}{\nrm u2}\)\,,
\]
where
\[
\chi_k^{(1)}(x):=\frac{\alpha_k}{1-(p-1)^{\alpha_k}}\,(1-x^2)\,,\quad
\chi_k^{(2)}(x):=\frac{\alpha_k}{1-(p-1)^{\gamma\alpha_k/2}}\,(1-x^\gamma)\,,
\]
and
\[
\chi_k^{(3)}(x):=\frac2{\log(p-1)}\,\log x\,.
\]
Here $x\in(0,1]$ and the proof of Corollary~\ref{Cor:Spectral3} amounts to
\[
\chi_k^{(2)}(x)\le\max\left\{\chi_k^{(1)}(x),\chi_k^{(3)}(x)\right\}\,.
\]
This completes the proof of Theorem \ref{Thm:Main2}. See Fig.~\ref{Fig5-2}.
%*********************************************************************
\begin{figure}[ht]
\includegraphics[width=8cm]{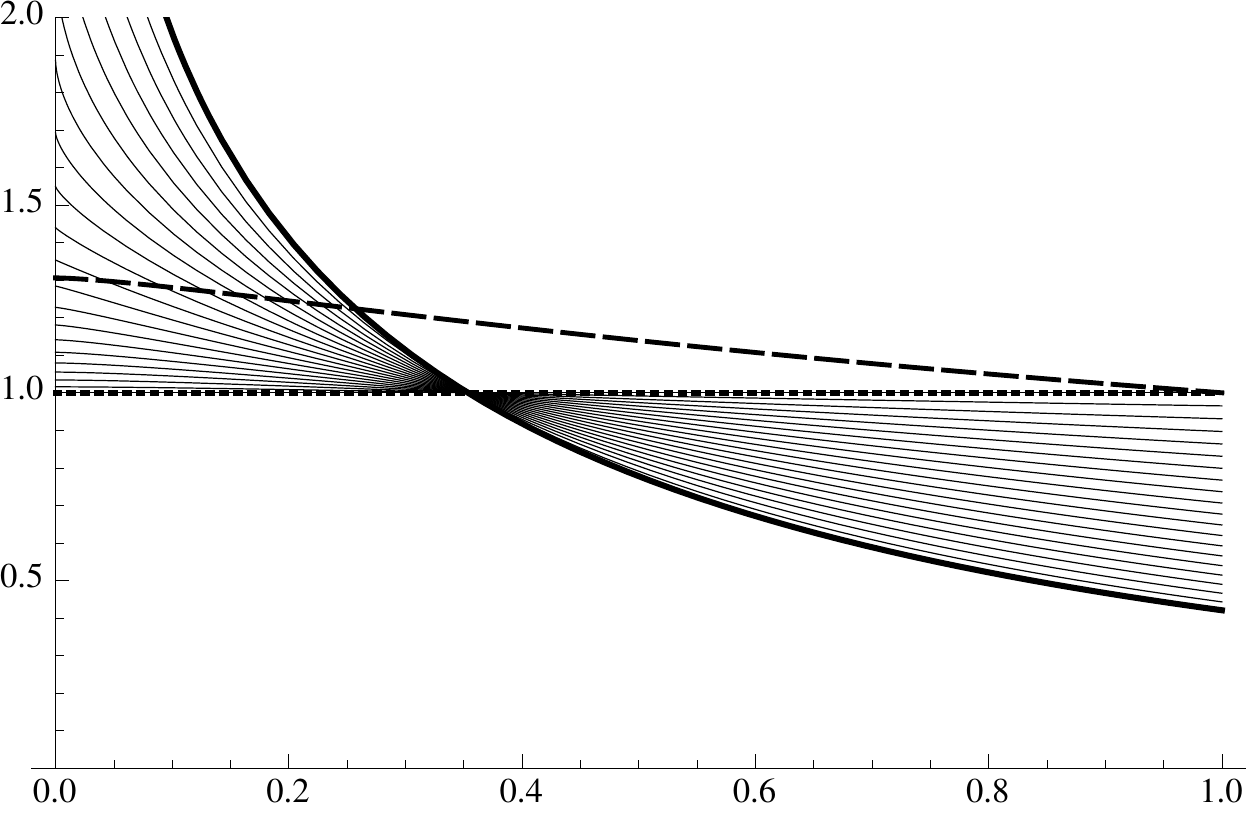}
\caption{\small\sl\label{Fig5-2} Plot of $\xi\mapsto\chi_k^{(2)}(\xi)/\chi_k^{(1)}(\xi)$ for various values of $\gamma$ and of $\xi\mapsto\chi_k^{(3)}(\xi)/\chi_k^{(1)}(\xi)$ (bold curve) when $d=2$, $p=3/2$, $k=1$ and $\alpha_1=3$. The plot of Fig.~\ref{Fig5-1}-left appears as a dashed line. Notice that an additional assumption required in Theorem~\ref{Thm:Main2}, namely the orthogonality to the eigenfunctions associated with $\lambda_1=d$.}
\end{figure}
%*********************************************************************

%%%%%%%%%%%%%%%%%%%%%%%%%%%%%%%%%%%%%%%%%%%%%%%%%%%%%%%%%%%%%%%%%%%%%%
\subsection{A review of some results in the critical case \texorpdfstring{$p=2$}{p=2}}\label{Spectral:critical}

The question of improvements under orthogonality constraints is a topic which has been studied in various contexts. Although our method does not apply to $p\ge2$, for completeness let us mention a few results that are concerned with the critical case $p=2$, or with Onofri related inequalities in dimension less or equal than $2$.

Let us start with the case $d=1$. In \cite{MR952815,MR960228} (also see \cite{MR960229}), B.~Osgood, R.~Phillips and P.~Sarnak established the inequality, known as the first Lebedev-Millin inequality \cite[Section~5.1]{duren1983univalent},
\[
\log\int_0^{2\pi}e^\phi\,\frac{d\theta}{2\pi}\le\sum_{k=1}^\infty k\,|\hat\phi(k)|^2
\]
with $\int_0^{2\pi}\phi\,d\theta=0$, which can be improved into
\[
\log\int_0^{2\pi}e^\phi\,\frac{d\theta}{2\pi}\le\frac 12\sum_{k=1}^\infty k\,|\hat\phi(k)|^2
\]
if, additionally, $\int_0^{2\pi}e^\phi\,e^{i\theta}\,d\theta=0$. This inequality has been improved to
\[
\log\int_0^{2\pi}e^\phi\,\frac{d\theta}{2\pi}\le\frac 1{n+1}\sum_{k=1}^\infty k\,|\hat\phi(k)|^2
\]
under the conditions $\int_0^{2\pi}\phi\,d\theta=0$ and $\int_0^{2\pi}e^\phi\,e^{i m\theta}\,d\theta=0$ for any $m=1$, $2$\ldots $n$ by H.~ÊWidom in \cite{MR929019}.

The case $d=2$ is not understood as well. Consider the inequality
\[
\log\int_{\S^2}e^u\;\frac{d\omega}{4\pi}-\int_{\S^2}u\;\frac{d\omega}{4\pi}\le\frac\alpha{16\pi}\int_{\S^2}|\nabla u|^2\,d\omega\,.
\]
Here $d\omega$ is the measure induced by Lebesgue's measure on $\S^2\subset\R^3$ (without normalization). This inequality has been established in \cite{MR0301504} (without optimal constant) and in \cite{MR677001} with sharp constant $\alpha=1$ when $u$ is an arbitrary function in $\H^1(\S^2)$. In \cite{MR908146}, S.-Y. A.~Chang and P.C.~Yang asked the question whether the inequality is true with $\alpha=1/2$ if
\[
\int_{\S^2}e^u\,x_j\;d\omega=0\quad\forall\,j=1\,,\;2\,,\;3\,.
\]
A partial answer has been given in \cite{MR2670931} by N.~Ghoussoub and C.-S.~Lin, who showed that in such a case $\alpha<2/3$.

In dimension $d\ge3$, there are no explicit results, as far as we know, but G.~Bianchi and H.~Egnell show in \cite{MR1124290} that improvements (without optimal constants) can also be achieved if we require the appropriate orthogonality conditions. A more precise statement, although still not fully explicit in terms of spectral estimates, has been established in \cite{MR2538501}.

%%%%%%%%%%%%%%%%%%%%%%%%%%%%%%%%%%%%%%%%%%%%%%%%%%%%%%%%%%%%%%%%%%%%%%
\subsection{An open problem}\label{Sec:Open}

The results of Sections~\ref{Sec:Improved3} and \ref{Spectral:critical} raise a straightforward question: how can spectral estimates together with well-chosen orthogonality constraints provide improved inequalities in the subcritical range $p\in(2,2^*)$ when $d>2$, or in the range $p>2$ when $d\le2$ ?

%%%%%%%%%%%%%%%%%%%%%%%%%%%%%%%%%%%%%%%%%%%%%%%%%%%%%%%%%%%%%%%%%%%%%%
%%%%%%%%%%%%%%%%%%%%%%%%%%%%%%%%%%%%%%%%%%%%%%%%%%%%%%%%%%%%%%%%%%%%%%
\section{Csisz\'ar-Kullback-Pinsker inequalities}\label{Sec:CKP}

The standard form of the Csisz\'ar-Kullback-Pinsker inequality is
\[
\iom{f\,\log\(\frac f{f_0}\)}\ge\frac 1{4\,M}\,\nrmom{f-f_0}1^2\,,
\]
for any nonnegative integrable functions $f$ and $f_0$ such that \hbox{$\iS{\kern-1.6pt f\kern-1.6pt}=\iS{\kern-1.6pt f_0\kern-1.6pt}=M$}. See \cite{Csiszar67,Kullback67,MR0213190} for details. Here we need a generalized form of this inequality for $p\neq2$. As will be made clear in the proof, the result is based on Taylor expansions and does depend neither on the domain of integration $\Omega$ nor on the positive measure $d\mu$. We shall therefore assume that \emph{$\Omega$ is a measurable subset in a submanifold of the Euclidean space, and that $d\mu$ is a probability measure on $\Omega$}, without further notice. In practice $(\Omega,d\mu)$ is either the $d$-dimensional sphere $\S^d$ endowed with the probability measure induced by Lebesgue's measure, or $(-1,1)$ and $d\mu=d\nu_d$ corresponds to the setting of the ultraspherical operator.

The following result is somewhat standard and the interested reader is invited to read for instance \cite{MR1951784,MR1801751}. It allows us to prove various kinds of Csisz\'ar-Kullback-Pinsker type inequalities. For completeness, we also provide a proof.
%-------------------------------------------------------------------------------------
\begin{proposition}\label{prop:CK0} Assume that $q\in[1,2]$ and consider $\psi(s)=s^q$ if $q>1$ or $\psi(s)=s\,\log s$ if $q=1$. Let $f$ and $g$ be two nonnegative functions in $\mathrm L^1\cap\mathrm L^q(\Omega,d\mu)$. Then
\begin{eqnarray}\label{CKgen}
e_\psi[\,f|g\,]&:=&\iom{\Big[\psi(f)-\psi(g)-\psi'(g)(f-g)\Big]}\nonumber\\
&&\geq \frac{q\,(q-1)}{2^{2/q}} \,\min\,\(\nrmom fq^{q-2},\nrmom gq^{q-2}\)\nrmom{f-g}q^2\,.
\end{eqnarray}\end{proposition}
%-------------------------------------------------------------------------------------
\begin{remark} The case $q=1$ is the well-known Csisz\'ar-Kullback-Pinsker inequality, \emph{cf.} \cite{Csiszar67,Kullback67,MR0213190,MR1842428}. The case $q=2$ is a consequence of the expansion of the square.
\end{remark}

\begin{proof} Assume first that $f>0$. By a Taylor expansion at order two, we get
\be{taylor}
e_\psi[\,f|g\,]=\frac12\iom{\psi''(\xi)\,|f-g|^2}\ge\frac A2\iom{\xi^{q-2}\,|f-g|^2}\,,
\ee
where $\xi$ lies between $f$ and $g$. By H\"older's inequality, for any $h>0$ and for any measurable set $\mathcal A\subset\Omega$, we get
\[
\int_{\mathcal A}|f-g|^q\,h^{-\alpha}\,h^\alpha\;d\mu\le\(\int_{\mathcal A}|f-g|^2\,h^{q-2}\;d\mu\)^{q/2}\(\int_{\mathcal A}
h^{\alpha s}\;d\mu\)^{1/s}
\]
with $\alpha=q\,(2-q)/2$, $s=2/(2-q)$. Thus,
\[
\(\int_{\mathcal A}|f-g|^2\,h^{q-2}\;d\mu\)^{q/2}\ge\int_{\mathcal A}|f-g|^q\;d\mu\;\(\int_{\mathcal A}h^q\;d\mu\)^{(q-2)/2}\,.
\]
We apply this formula to two different sets.
\begin{itemize}
\item[1)] On ${\mathcal A}={{\mathcal A}_1}=\{ x\in\Omega\,:\,f(x) > g(x)\}$, use $\xi^{q-2} > f^{q-2}$ and take $h=f$:
\[
\(\int_{{\mathcal A}_1}|f-g|^2\xi^{q-2}\;d\mu\)^{q/2}\ge\(\int_{{\mathcal A}_1}|f-g|^q\;d\mu\)\;\nrmom fq^{-(2-q)\,q/2}\,.
\]
\item[2)] On ${\mathcal A}={{\mathcal A}_2}=\{ x\in\Omega\,:\,f(x)\le g(x)\}$, use $\xi^{q-2}\geq g^{q-2}$ and take $h=g$:
\[
\(\int_{{\mathcal A}_2}|f-g|^2\xi^{q-2}\;d\mu\)^{q/2}\ge\(\int_{{\mathcal A}_2}|f-g|^q\;d\mu\)\;\nrmom gq^{-(2-q)\,q/2}\,.
\]\end{itemize}
To prove \eqref{CKgen} in the case $f>0$, we just add the two previous inequalities in \eqref{taylor} and use the inequality $(a+b)^r\le 2^{r-1}(a^r+b^r)$ for any $a,b\geq 0$ and $r\geq 1$. To handle the case $f\ge0$, we proceed by a density argument and conclude by using Lebesgue's convergence theorem.\end{proof}

Next we give the proof of some generalized Csisz\'ar-Kullback-Pinsker inequalities for various values of $p$, all easily derived from the above proposition.
The case of $p\in[1,2)$ can be handled with $q=2/p$. For $p\in(2,4]$, one can use the inequality of Proposition~\ref{prop:CK0} written for $q=p/2$ and control $\nrm{u^2-\bar u^2}q$. For $p\ge4$, the control is achieved in terms of $\nrm{u^{p-2}-\bar u^{p-2}}{p/(p-2)}$. For each range, the average $\bar u$ has to be defined specifically. We do not claim originality for the following result, as it has probably been discovered in other settings. Let us just mention a few additional references: in the case $p>2$, see: \cite{MR1777035,MR1940370} and \cite[Proposition 2]{MR2295184} for related results. A recent contribution in a similar spirit can be found in \cite{CFL}.
%-------------------------------------------------------------------------------------
\begin{corollary}\label{cor:CK} For any $u\in\mathrm L^1\cap\mathrm L^p(\Omega,d\mu)$, we have
\[\begin{array}{ll}
\nrmom u2^2-\nrmom up^2\ge\kappa(p)\,\nrmom u2^{2\,(1-p)}\,\nrmom{u^p-\bar u^p}{2/p}^2&\kern -13pt\mbox{if}\;p\in[1,2)\,,\\
\iom{|u|^2\,\log\Big(\tfrac{|u|^2}{\nrmom u2^2}\Big)}\ge\frac{2\,|\kappa'(2)|}{\nrmom u2^2}\,\nrmom{u^2-\bar u^2}1^2&\kern -13pt\mbox{if}\;p=2\,,\\
\nrmom up^2-\nrmom u2^2\ge\kappa(p)\,\nrmom up^{-2}\,\nrmom{u^2-\bar u^2}{p/2}^2&\kern -13pt\mbox{if}\;p\in(2,4)\,,\\
\nrmom up^2-\nrmom u2^2\ge\kappa(p)\,\nrmom up^{2\,(3-p)}\,\nrmom{u^{p-2}-\bar u^{p-2}}{\frac p{p-2}}^2&\kern -13pt\mbox{if}\;p\ge4\,,\\
\end{array}\]
where
\[\begin{array}{lll}
\kappa(p)=2^{1-p}\,\tfrac{2-p}{p^2}\,,\quad&\bar u=\nrmom up&\;\mbox{if}\;p\in[1,2)\,,\\
|\kappa'(2)|=\tfrac18\,,\quad&\bar u=\nrmom u2&\;\mbox{if}\;p=2\,,\\
\kappa(p)=2^{-1-\frac4p}\,(p-2)\,,\quad&\bar u=\nrmom u2&\;\mbox{if}\;p\in(2,4)\,,\\
\kappa(p)=2^\frac4p\,(p-2)^{-2}\,,\quad&\bar u=\nrmom u{p-2}&\;\mbox{if}\;p\ge4\,.
\end{array}\]
\end{corollary}
%-------------------------------------------------------------------------------------
\begin{proof} When $p\in[1,2)$, we apply Proposition~\ref{prop:CK0} with $q=2/p$, $\psi(s):=s^{2/p}$, to $f=u^p$ and $g=\nrmom up^p$. The result follows using $\nrmom up\le\nrmom u2$. If $p=2$, Proposition~\ref{prop:CK0} directly applies, with $\psi(s):=s\,\log s$. If $p\in(2,4)$, we apply Proposition~\ref{prop:CK0} with $q=p/2$, $\psi(s):=s^{p/2}$, to $f=u^2$ and $g=\nrmom{u^2}1=\bar u^2$, so that
\[
\nrmom up^p-\nrmom u2^p\ge2^{-2-\frac4p}\,p\,(p-2)\,\nrmom up^{p-4}\,\nrmom{u^2-\bar u^2}{p/2}^2\,.
\]
Since $t\mapsto1-t-\frac2p\,(1-t^{p/2})$ is convex, nonnegative, with $t=\nrmom u2^2/\nrmom up^2$, we can write that
\begin{multline*}
\nrmom up^2-\nrmom u2^2\\
=\nrmom up^2\,(1-t)\ge\nrmom up^2\,\frac2p\,(1-t^{p/2})\\
=\frac2p\,\nrmom up^{2-p}\,\Big[\nrmom up^p-\nrmom u2^p\Big]
\end{multline*}
and get the announced result.

If $p\ge4$, we apply Proposition~\ref{prop:CK0} with $q=p/(p-2)$, $\psi(s):=s^{p/(p-2)}$, to $f=u^{p-2}$ and $g=\nrmom{u^{p-2}}1=\bar u^{p-2}$, so that
\[
\nrmom up^p-\bar u^p\ge2^{-1+\frac4p}\,\frac p{(p-2)^2}\,\nrmom{u^{p-2}}{p/(p-2)}^{\frac p{p-2}-2}\,\nrmom{u^{p-2}-\bar u^{p-2}}{\frac p{p-2}}^2\,.
\]
We conclude by using the convexity of $t\mapsto1-t-\frac2p\,(1-t^{p/2})\,$ with $\,t=\bar u^2/\nrmom up^2$ and, by H\"older's inequality, the fact that $\bar u=\nrmom u{p-2}\ge\nrmom u2$.
\end{proof}
Proposition~\ref{prop:CK} is a straightforward consequence of Corollary~\ref{cor:CK}. Details are left to the reader.

\begin{remark} In the proof of Corollary~\ref{cor:CK}, $p=4$ is a threshold case. In \cite[Lemma 2]{MR2295184}, it is $p=3$ that plays a special role. Many other estimates can be derived for $p>2$ and the value $p=4$ \emph{a priori} plays no special role, as it is shown by the following computation. Let
\[
f(w):=w^p-1-\frac p{p-p_0}\,(w^{p-p_0}-1)-\frac1{p-1}\,|w-1|^p\,.
\]
Two differentiations show that
\[\begin{array}{l}
f'(w)=p\,w^{p-2}(w-w^{1-p_0})-\frac p{p-1}\,|w-1|^{p-2}(w-1)\,,\\[6pt]
f''(w)=p\,w^{p-3}\big((p-1)(w-1)+p_0\,w^{1-p_0}\big)-p\,|w-1|^{p-2}\,.
\end{array}\]
On the one hand we have $f(1)=f'(1)=0$ and, on the other hand,
\[
\frac 1p\,f''(w)\geq (w-1)^{p-3}\,\big((p-1)(w-1)+p_0\,w^{1-p_0}\big)-(w-1)^{p-2}\geq
(p-2)\,(w-1)^{p-2}
\]
for any $w\geq 1$ if we assume that
\[
p\ge3\quad\mbox{and}\quad p_0\ge1\,.
\]
Thus $f$ is convex and therefore nonnegative on $(1,+\infty)$.

Now, if we define $\bar u^{p-p_0}=\iom{u^{p-p_0}}$, by integrating the inequality $f(u/{\bar u})\ge 0$ with respect to the measure $\bar u^p\,d\mu$, we arrive at
\[
\iom{u^p}-\bar u^p\ge\frac1{p-1}\iom{|u-\bar u|^p}\,.
\]
Hence, if $p-p_0\ge2$, then the inequality $\bar u=\nrmom u{p-p_0}\ge\nrmom u2$ and the convexity of $t\mapsto1-t-\frac2p\,(1-t^{p/2})$ applied with $t=\bar u^2/\nrmom up^2$ allows us to conclude that
\begin{multline*}
\nrmom up^2-\nrmom u2^2\ge\frac2p\,\nrmom up^{2-p}\,\Big[\nrmom up^p-\nrmom u2^p\Big]\\
\ge\frac2{p\,(p-1)}\,\nrmom up^{2-p}\,\nrmom{u-\bar u}p^p
\end{multline*}
with $\bar u=\nrmom u{p-p_0}$.
\end{remark}

%%%%%%%%%%%%%%%%%%%%%%%%%%%%%%%%%%%%%%%%%%%%%%%%%%%%%%%%%%%%%%%%%%%%%%
%%%%%%%%%%%%%%%%%%%%%%%%%%%%%%%%%%%%%%%%%%%%%%%%%%%%%%%%%%%%%%%%%%%%%%
\appendix\section{A discussion on the range of admissible \texorpdfstring{$p$ and $\beta$}{p and beta}}\label{Sec:Appendix}

Consider $\gamma(\beta)=-\,\big(\frac{d-1}{d+2}\big)^2(\kappa+\beta-1)^2+\kappa\,(\beta-1)+(\kappa+\beta-1)\,\frac d{d+2}$ with $\kappa=\beta\,(p-2)+1$ as in~\eqref{Eqn:gamma}. Denoting
\[
\mathsf a(p,d):=2-p+\left[\frac{(d-1)\,(p-1)}{d+2}\right]^2\quad\mbox{and}\quad\mathsf b(p,d):=\frac{d+3-p}{d+2}\,,
\]
we have that
\[
-\,\gamma(\beta)=\mathsf a\,\beta^2-2\,\mathsf b\,\beta+1\,.
\]
Provided $\mathsf a\neq 0$ (this will be discussed below), the two roots of the equation $\gamma(\beta)=0$ are
\[
\beta_-:=\min\frac{\mathsf b\pm\sqrt{\mathsf b^2-\mathsf a}}{2\,\mathsf a}\quad\mbox{and}\quad\beta_+:=\max\frac{\mathsf b\pm\sqrt{\mathsf b^2-\mathsf a}}{2\,\mathsf a}\,.
\]
%---------------------------------------------------------------------
\begin{lemma}\label{Lem:beta} With the above definitions, $\gamma$ is nonnegative if and only if:
\begin{enumerate}
\item $d\ge5$ and $\beta\in[\beta_-(p,d),\beta_+(p,d)]$.
\item $d=4$ and\\
$\bullet$ either $p\in[1,3)\cup(3,4]$ and $\beta\in[\beta_-(p,4),\beta_+(p,4)]$,\\
$\bullet$ or $p=3$ and $\beta\ge3/4$.
\item $d=3$ and\\
$\bullet$ either $p\in[1,9/4)$ and $\beta\in[\beta_-(p,3),\beta_+(p,3)]$,\\
$\bullet$ or $p\in(9/4,6)$ and $\beta\in(-\infty,\beta_-(p,3)]\cup[\beta_+(p,3),+\infty)$,\\
$\bullet$ or $p=9/4$ and $\beta\ge2/3$.
\item $d=2$ and\\
$\bullet$ either $p\in[1,9-4\,\sqrt3)\cup(9+4\,\sqrt3,\infty)$ and $\beta\in[\beta_-(p,2),\beta_+(p,2)]$,\\
$\bullet$ or $p\in(9-4\,\sqrt3,9+4\,\sqrt3)$ and $\beta\in(-\infty,\beta_-(p,2)]\cup[\beta_+(p,2),+\infty)$,\\
$\bullet$ or $p=9-4\,\sqrt3$ and $\beta\ge(1+\sqrt 3)/4$,\\
$\bullet$ or $p=9+4\,\sqrt3$ and $\beta\le(1-\sqrt 3)/4$.
\item $d=1$ and\\
$\bullet$ either $p\in[1,2)$, and $\beta\in[\beta_-(p,1),\beta_+(p,1)]$,\\
$\bullet$ or $p=2$ and $\beta\ge3/4$,\\
$\bullet$ or $p>2$ and $\beta\in(-\infty,\beta_-(p,1)]\cup[\beta_+(p,1),+\infty)$.
\end{enumerate}
\end{lemma}
%---------------------------------------------------------------------
See Fig.~\ref{Fig2}, and Fig.~\ref{Fig3-2} for the special case $d=2$.

%*********************************************************************
\begin{figure}[ht]
\includegraphics[width=6cm]{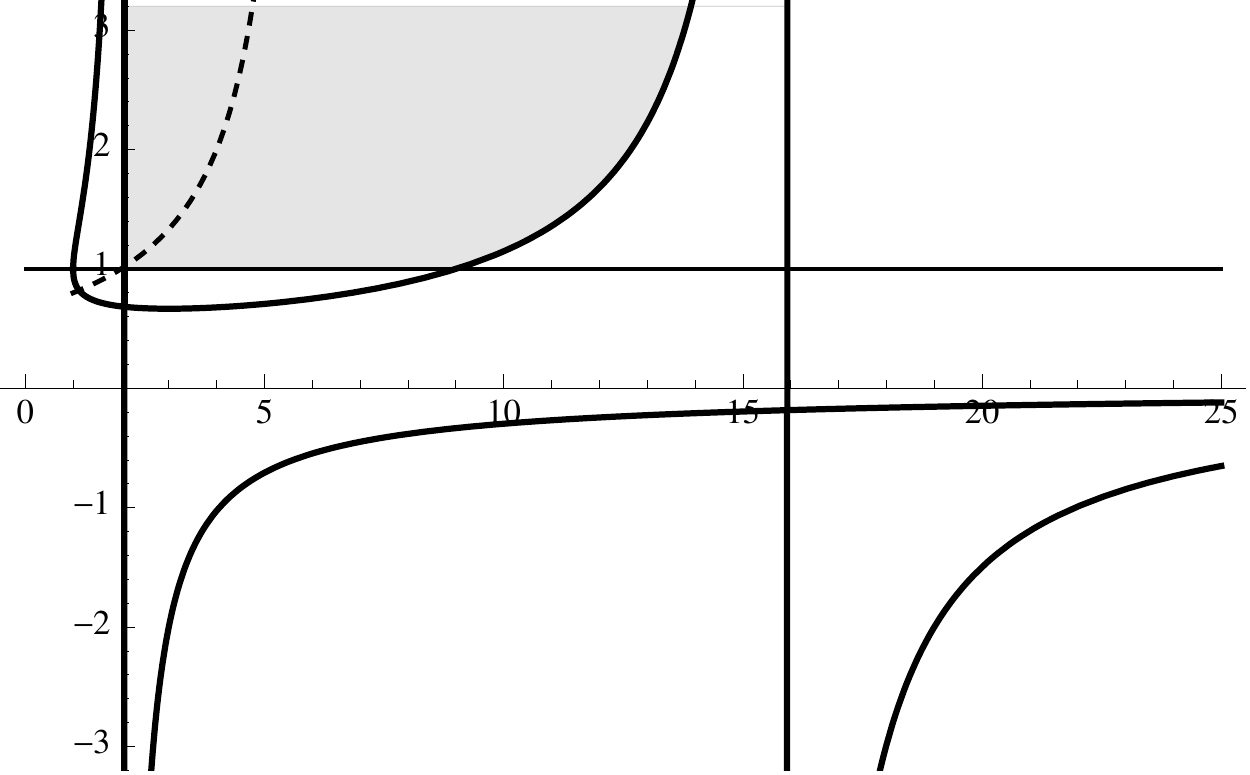}\hspace*{1cm}
\includegraphics[width=4cm,height=4cm]{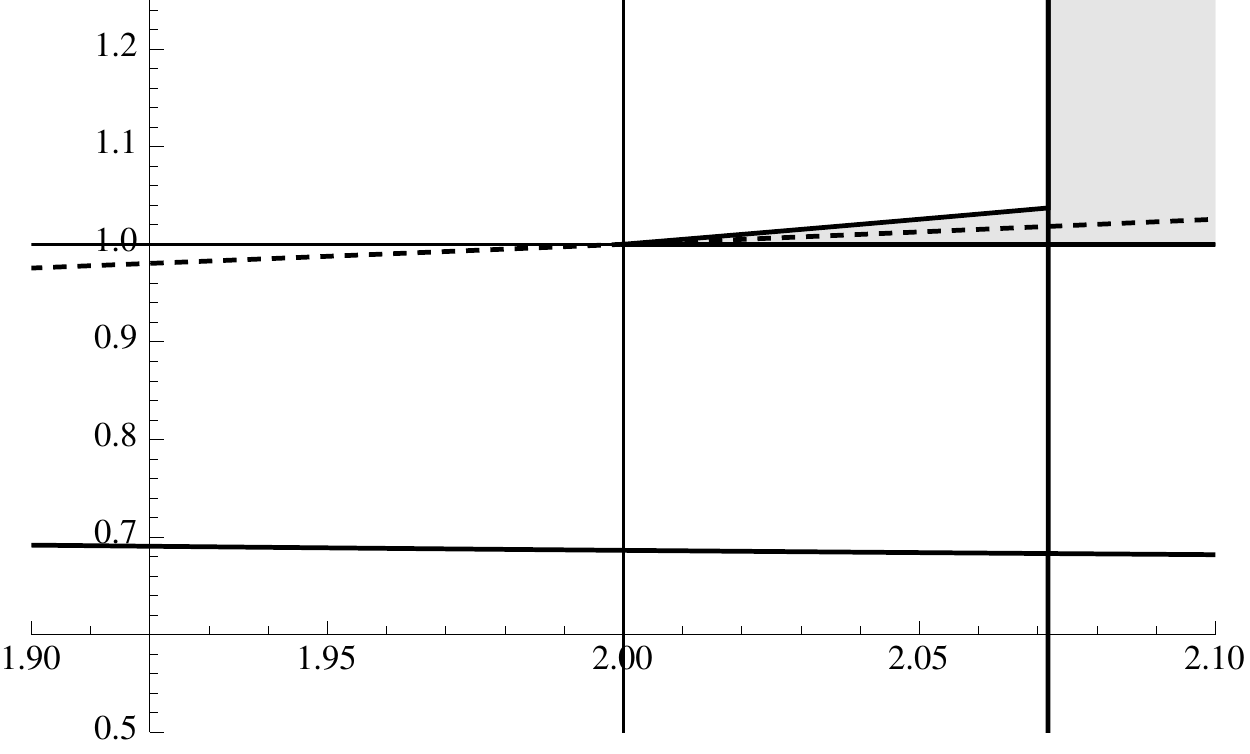}
\caption{\small\sl The case of the dimension $d=2$ deserves a special treatment. The coefficient $\mathsf a$ is negative in the interval $(9-\,4\sqrt 3,9+\,4\sqrt 3)\approx(2.0718,15.9282)$. The right plot is an enlargement of the left plot in a neighborhood of $p=2$, $\beta=1$. The grey area corresponds to the interior of $\mathfrak B(p,2)$\label{Fig3-2}}
\end{figure}
%*********************************************************************

\begin{proof} First of all, we observe that $\gamma(\beta)=-\,(\beta-1)^2$ if $p=1$ and $\gamma(\beta)=-\,((d-3)\,\beta+2-d)^2/(d-2)^2$ if $p=2^*$, which is consistent with the fact that $\beta_\pm(1,d)=1$ and $\beta_\pm(2^*,d)=\frac{d-2}{d-3}$ if $d\ge4$. Notice that for $d=3$ and $p=6=2^*$, $\gamma(\beta)=-1$ is independent of $\beta$, and hence always negative.

Elementary computations show that
\[
\mathsf b^2-\mathsf a=\frac{d\,(p-1)}{(d-2)\,(d+2)^2}\(\frac{2\,d}{d-2}-p\)\;\mbox{if}\;d\neq2\quad\mbox{and}\quad\mathsf b^2-\mathsf a=\frac{p-1}2\;\mbox{if}\;d=2
\]
is positive if and only if either $p\in(1,2^*)$ and $d\neq2$ or $p>1$ and $d=2$. If $\mathsf a\neq0$ and if $\beta_\pm$ are the two roots of the equation $\gamma(\beta)=0$, then $\gamma(\beta)$ is positive if and only if one of the following conditions is satisfied
\begin{enumerate}
\item $\mathsf a$ is positive and $\beta\in(\beta_-(p),\beta_+(p))$,
\item $\mathsf a$ is negative and $\beta\in(-\infty,\beta_-(p))\cup(\beta_+(p),+\infty)$,
\item $\mathsf a=0$, $\mathsf b$ is positive and $\beta>1/(2\,\mathsf b)$,
\item $\mathsf a=0$, $\mathsf b$ is negative and $\beta<1/(2\,\mathsf b)$.
\end{enumerate}

Since
\[
(d+2)^2\,\mathsf a=(d-1)^2\,p^2-3\,(d^2+2)\,p+3\,(d^2+2\,d+3)
\]
we find that the discriminant $9\,(d^2+2)^2-12\,(d^2+2\,d+3)\,(d-1)^2=3\,(4-d)\,d\,(d+2)^2$ is negative for any $d\ge5$, but nonnegative if $d=1$, $2$, $3$, or $4$. In the range $d\in(0,1)\cup(1,4)$, the equation $\mathsf a(p,d)=0$ has two roots
\[
p_\pm(d):=\frac{3\,(d^2+2)\pm(d+2)\,\sqrt{3\,d\,(4-d)}}{2\,(d-1)^2}\,.
\]

%*********************************************************************
\begin{figure}[hb]
\includegraphics[height=4.5cm,width=6cm]{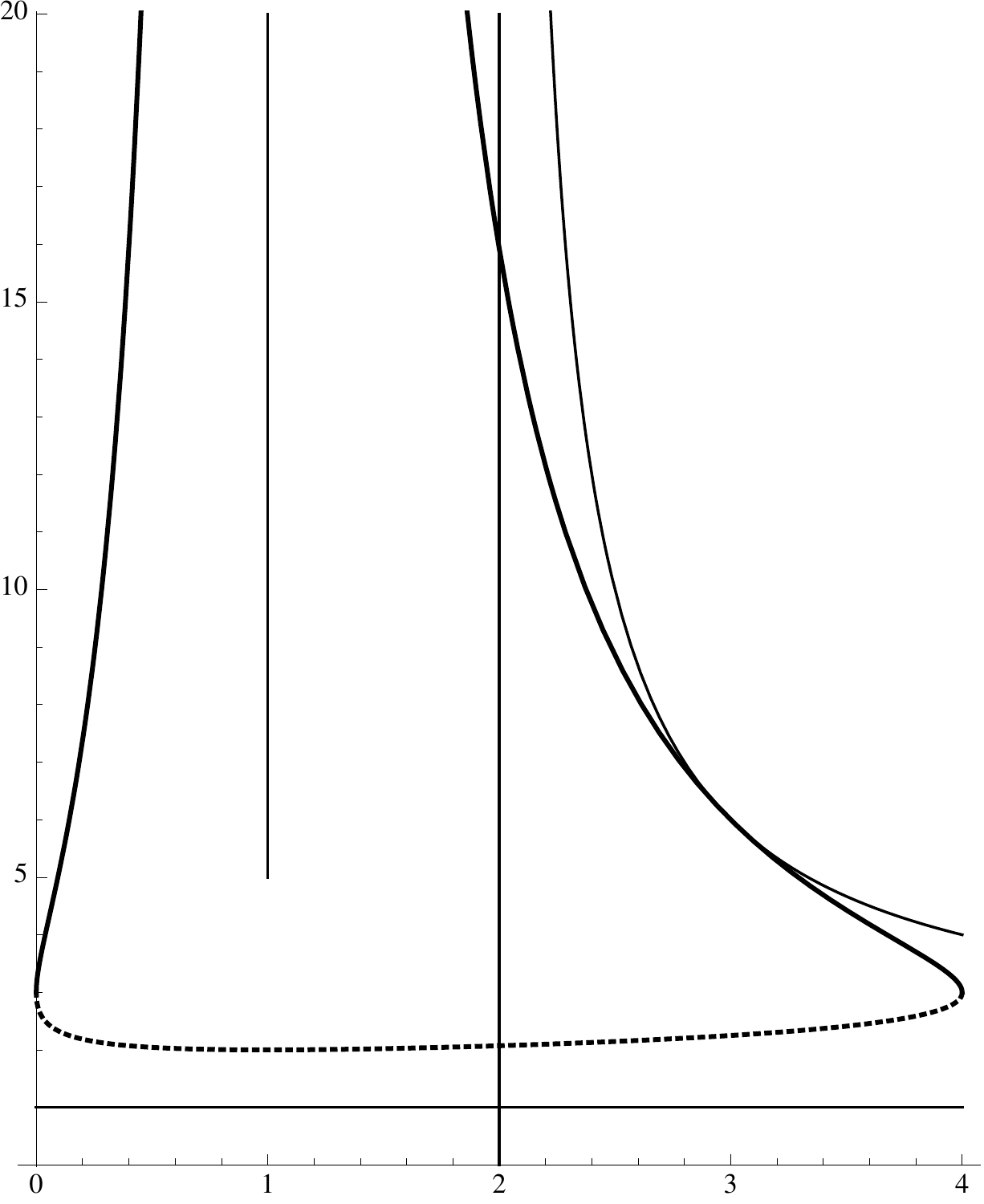}
\caption{\small\sl Horizontal axis: $d$, vertical axis: $p$. Bold plain curve: $d\mapsto p_+(d)$, bold dotted curve: $d\mapsto p_-(d)$. Other curves shown on the picture correspond to $d=1$, $d=2$ and $p=\frac{2\,d}{d-2}$. We notice that $\mathsf a(p,d)<0$ if and only if either $d=1$ and $p\in(2,\infty)$, or $d=2$ and $p\in(9-4\,\sqrt3,9+4\,\sqrt3)$, or $d=3$ and $p\in(\frac94,6)$. We may also notice that $\lim_{d\to1}p_-(d)=2$ and $\lim_{d\to1}p_+(d)=+\infty$ (if we consider $p_\pm$ as a function of $d\in\R$).\label{Fig1}}
\end{figure}
%*********************************************************************
\begin{table}[ht]
\begin{tabular}{|c||c|c|c|c|c|}
\hline
$d$&$1$&$2$&$3$&$4$&$\ge5$\cr
\hline\hline
&&&&&\cr
$\mathsf a$&$2-p$&$\tfrac{p^2-18\,p+33}{16}$&$\displaystyle\tfrac{4\,p^2-33\,p+54}{25}$&$\displaystyle\tfrac{(p-3)^2}4$&$\displaystyle2-p+\left[\tfrac{(d-1)\,(p-1)}{d+2}\right]^2>0$\cr&&&&&\cr
\hline
&&&&&\cr
$\mathsf b$&$\displaystyle\tfrac{4-p}3$&$\displaystyle\tfrac{5-p}4$&$\displaystyle\tfrac{6-p}5$&$\displaystyle\tfrac{7-p}6$&$\displaystyle\tfrac{d+3-p}{d+2}$\cr
&&&&&\cr
\hline
&&&&&\cr
$p_-$&not defined&$9-4\,\sqrt3$&$\displaystyle\tfrac94$&$3$&not defined\cr
&&&&&\cr
\hline
&&&&&\cr
$p_+$&not defined&$9+4\,\sqrt3$&$6$&$3$&not defined\cr
&&&&&\cr
\hline
\end{tabular}
\vspace*{6pt}
\caption{Values of $\mathsf a$ and $\mathsf b$ for $d\in\N$, $d\ge1$.}
\end{table}

See Fig.~\ref{Fig1} for a plot and Table~1 for a summary of the values of $\mathsf a$ and $\mathsf b$ depending on the value of $d$, when $d$ is an integer.

Hence we know that $\mathsf a=0$ if and only if
\begin{enumerate}
\item $d=1$ and $p=2$,
\item $d=2$ and $p=p_\pm(2)=9\pm4\,\sqrt3\approx9\pm6.8292$,
\item $d=3$ and $p=p_-(3)=9/4$ or $p=p_+(3)=6=2\,d/(d-2)$,
\item $d=4$ and $p=3$.
\end{enumerate}
In these cases, the sign of $\mathsf b=\frac{d+3-p}{d+2}$ matters:
\begin{enumerate}
\item if $d=1$ and $p=2$ then $\mathsf b=\frac23>0$,
\item if $d=2$ and $p=p_-(2)=9-4\,\sqrt3$ then $\mathsf b=\sqrt3-1>0$; if $d=2$ and $p=p_+(2)=9+4\,\sqrt3$ then $\mathsf b=-1-\sqrt3<0$,
\item if $d=3$ and $p=p_-(3)=9/4$ then $\mathsf b=\frac34>0$; if $d=3$ and $p=p_+(3)=6=2\,d/(d-2)$ then $\mathsf b=0$ (but then $\mu(\beta)=1$ is always positive),
\item if $d=4$ and $p=3$ then $\mathsf b=\frac23>0$.
\end{enumerate}

We also get that $\mathsf a$ is negative if and only if\\
$\bullet$ either $d=1$ and $p\in(2,\infty)$,\\
$\bullet$ or $d=2$, $3$ and $p\in(p_-(d),p_+(d))$.

Similarly $\mathsf a$ is positive if and only if\\
$\bullet$ either $d=1$ and $p\in[1,2)$,\\
$\bullet$ or $d=2$ and $p\in[1,9-4\,\sqrt3)\cup(9+4\,\sqrt3,\infty)$,\\
$\bullet$ or $d=3$ and $p\in[1,\frac94)$,\\
$\bullet$ or $d=4$ and $p\neq3$,\\
$\bullet$ or $d\ge5$.\\
Consistently, we may notice that $\mathsf a(1,d)\equiv 1$ for any $d\in\R\setminus\{1\}$ and, for any $d>2$, $\mathsf a(2^*,d)=\big(\frac{d-3}{d-2}\big)^2$ is positive unless $d=3$.

This concludes the proof by discussing the cases depending whether $\mathsf a=0$ (and the range of $\beta$ is determined by $\mathsf b)$ or $\mathsf a$ has a strict sign and $\beta_\pm$ defines the admissible range for $\beta$ in order that $\mu$ is nonpositive.
\end{proof}

\section*{Acknowledgments} J.D.~and M.J.E.~have been partially supported by the ANR grant \emph{NoNAP}. Part of this work was completed during M.K.'s visit at Ceremade, Universit\'e Paris-Dauphine. J.D.~participates in the AmSud \emph{QUESP} project and thanks for support. J.D.~has also been partially supported by ANR grants \emph{STAB} and \emph{Kibord},  and the ECOS project \emph{C11E07}. M.K.~has been partially supported by the FONDECYT grant 1130126, the ECOS project \emph{C11E07} and \emph{Fondo Basal CMM}. M.L.~has been partially supported by the NSF grant DMS-1301555. J.D.~thanks the organizers of the Conference on \emph{Nonlinear Elliptic and Parabolic Partial Differential Equations} held in Milano in June 2013, which has provided the opportunity for completing this paper.
%%%%%%%%%%%%%%%%%%%%%%%%%%%%%%%%%%%%%%%%%%%%%%%%%%%%%%%%%%%%%%%%%%%%%%
%%%%%%%%%%%%%%%%%%%%%%%%%%%%%%%%%%%%%%%%%%%%%%%%%%%%%%%%%%%%%%%%%%%%%%

\medskip
\noindent{\sl\small \copyright~2013 by the authors. This paper may be reproduced, in its entirety, for non-commercial purposes.}

%%%%%%%%%%%%%%%%%%%%%%%%%%%%%%%%%%%%%%%%%%%%%%%%%%%%%%%%%%%%%%%%%%%%%%
%%%%%%%%%%%%%%%%%%%%%%%%%%%%%%%%%%%%%%%%%%%%%%%%%%%%%%%%%%%%%%%%%%%%%%
%\nocite*
\def\refname{REFERENCES}
%\bibliographystyle{siam}
%\bibliography{Super}

\medskip
% The data information below will be filled by AIMS editorial staff
Received September  2013; revised December 2013.
\medskip
\end{document}